\documentclass[11pt]{article}
\usepackage[pdftex]{graphicx}
\usepackage{bm}
\usepackage{amsmath,amsthm,amssymb,url}
\usepackage{multirow,bigdelim}
\usepackage{amscd}

\newtheorem{thm}{Theorem}[section]
\newtheorem{theorem}[thm]{Theorem}
\newtheorem{proposition}[thm]{Proposition}
\newtheorem{lemma}[thm]{Lemma}

\newtheorem{claim}[thm]{Claim}

\DeclareMathOperator{\rank}{rank}

\newcommand{\id}{\mathrm{id}}

\newcommand{\aff}{\mathrm{aff}}

\newcommand{\M}{{\mathcal{M}}}

\usepackage{tikz}
\usetikzlibrary{arrows}

\newcommand{\myGlobalTransformation}[2]
{
    \pgftransformcm{1}{0}{0.4}{0.5}{\pgfpoint{#1cm}{#2cm}}
}

\newcommand{\gridThreeD}[3]
{
    \begin{scope}
        \myGlobalTransformation{#1}{#2};
        \draw [#3,step=2cm] grid (8,8);
				\fill[white] (-2,0) -- (10,0) -- (10,-2) -- (-2,-2) -- (-2,0);
				\fill[white] (-2,-2) -- (-2,10) -- (0,10) -- (0,-2) -- (-2,-2);
				\fill[white] (-2,10) -- (10,10) -- (10,8) -- (-2,8) -- (-2,10);
				\fill[white] (8,10) -- (10,10) -- (10,-2) -- (8,-2) -- (8,10);
    \end{scope}
}

\title{Global Rigidity of Periodic Graphs under Fixed-lattice Representations}
\author{Vikt\'oria E. Kaszanitzky\thanks{
 Department of Mathematics and Statistics,
Lancaster University,
Lancaster, LA1 4YF, United Kingdom, and Department of Computer Science and Information Theory, Budapest University of Technology and Economics, Magyar tud\'osok krt 2., Budapest, 1117, Hungary
(\texttt{viktoria@cs.elte.hu}).
  Supported by EPSRC First Grant EP/M013642/1 and by the Hungarian Scientific Research Fund (OTKA, grant number K109240).}
\and
Bernd Schulze \thanks{
Department of Mathematics and Statistics,
Lancaster University,
Lancaster,
LA1 4YF, United Kingdom
(\texttt{b.schulze@lancaster.ac.uk}).
 Supported by EPSRC First Grant EP/M013642/1.}
\and
 Shin-ichi Tanigawa\thanks{Research Institute for Mathematical Sciences, Kyoto University, Sakyo-ku, Kyoto 606-8502,
Japan, and Centrum Wiskunde \& Informatica (CWI), Postbus 94079, 1090 GB Amsterdam, The
Netherlands (\texttt{tanigawa@kurims.kyoto-u.ac.jp}).
  Supported by JSPS Postdoctoral Fellowships for Research Abroad, JSPS Grant-in-Aid for Scientific Research(A)(25240004), and JSPS Grant-in-Aid for Young Scientists (B) (No. 15K15942).
}
}

\begin{document}


\maketitle

\begin{abstract}
In \cite{hend} Hendrickson proved that $(d+1)$-connectivity and redundant rigidity are necessary conditions for a generic (non-complete) bar-joint framework to be globally rigid in $\mathbb{R}^d$. Jackson and Jord\'{a}n~\cite{jj} confirmed that these conditions are also sufficient in $\mathbb{R}^2$, giving a combinatorial characterization of graphs whose generic realizations in $\mathbb{R}^2$ are globally rigid.
 In this paper, we  establish analogues of these results for infinite periodic frameworks under fixed lattice representations.
  Our combinatorial characterization of globally rigid generic periodic frameworks in $\mathbb{R}^2$ in particular implies toroidal and cylindrical counterparts of the theorem by Jackson and Jord\'{a}n.
\end{abstract}

\noindent {\em Key words}: global rigidity, periodic framework, toroidal framework, cylindrical framework, group-labeled graph, matroid connectivity.

\section{Introduction}
\label{sec:intro}

A bar-joint framework (or simply framework) in $\mathbb{R}^d$ is a pair $(G, p)$, where $G = (V, E)$ is a  graph and $p:V\to \mathbb{R}^d$ is a map. We think of a  framework  as a straight-line realization of $G$ in
$\mathbb{R}^d$ in which the length of an edge $uv\in E$ is given by the Euclidean distance
between the points $p(u)$ and $p(v)$.
A well-studied problem in discrete geometry is to determine the rigidity of frameworks.
A framework $(G,p)$ is called \emph{(locally) rigid} if, loosely speaking, it cannot be deformed continuously into another non-congruent framework while maintaining the lengths of all edges. 
It is well-known that a generic framework $(G,p)$ is  rigid in $\mathbb{R}^d$ if and only if \emph{every} generic realization of $G$ in $\mathbb{R}^d$ is  rigid~\cite{asiroth,gluck}. 
In view of this fact, a graph $G$ is said to be {\em  rigid} in $\mathbb{R}^d$ if some/any generic realization of $G$ is  rigid in $\mathbb{R}^d$.
A classical theorem by Laman~\cite{Lamanbib} says that $G$ is rigid in $\mathbb{R}^2$ if and only if $G$ contains a spanning subgraph $H$ with $|E(H)|=2|V(H)|-3$ such that
$|F|\leq 2|V(F)|-3$ for every nonempty $F\subseteq E(H)$, where $V(F)$ denotes the set of vertices incident to $F$.

Another central property in rigidity theory is global rigidity.
A  framework $(G,p)$ is called \emph{globally rigid}, if every  framework $(G,q)$ in $\mathbb{R}^d$ with the same edge lengths as $(G,p)$ has the same distances between all pairs of vertices as $(G, p)$.  Although deciding the global rigidity of a given framework is a difficult problem  in general, 
the problem becomes tractable if we restrict attention to \emph{generic} frameworks~\cite{gortler2010}, i.e. frameworks with the property that the coordinates of all points $p(v), v \in V$, are algebraically independent over $\mathbb{Q}$.

In 1992 Hendrickson~\cite{hend} established the following necessary condition for a generic framework  in $\mathbb{R}^d$ to be globally rigid.  
\begin{theorem} [Hendrickson~\cite{hend}] \label{thm:hendrickson} If $(G,p)$ is a generic globally rigid framework in $\mathbb{R}^d$, then $G$ is a complete graph with at most $d+1$ vertices, or $G$ is  $(d+1)$-connected and redundantly rigid in $\mathbb{R}^d$,
where $G$ is called redundantly rigid if $G-e$ is rigid for every $e\in E(G)$.
\end{theorem}

Although the converse direction is false for $d\geq 3$ as pointed out by Connelly, it turned out to be true for $d\leq 2$. 
\begin{theorem}[Jackson and Jord\'{a}n~\cite{jj}] \label{thm:jj}
Let $(G,p)$ be a generic framework in $\mathbb{R}^2$. Then $(G,p)$ is globally rigid if and only if  \(G\) is a complete graph with at most $3$ vertices, or  \(G\) is 3-connected and redundantly rigid in $\mathbb{R}^2$.
\end{theorem}

Based on the theory of  stress matrices by Connelly~\cite{connelly1982,connelly2005}, Gortler, Healy, and Thurston~\cite{gortler2010} gave an algebraic characterization of the global rigidity of generic frameworks. This in particular implies that all generic realizations of a given graph share the same global rigidity properties in $\mathbb{R}^d$, as in the case of local rigidity. 
However the problem of extending Theorem~\ref{thm:jj} to higher dimension remains  unsolved.

Largely motivated by practical applications in crystallography, materials science and engineering, as well as by mathematical applications in areas such as sphere packings, the  rigidity  analysis of infinite periodic frameworks  has seen an increased interest in recent years.  
Particularly relevant to our work is an extension of Laman's theorem to periodic frameworks with fixed-lattice representations by Ross~\cite{ross2}. 
In her rigidity model, a periodic framework can deform continuously under a fixed periodicity constraint (i.e., each orbit of points is fixed).

In this paper, we shall initiate the global rigidity counterpart of the rigidity theory of periodic frameworks. 
The global rigidity of periodic frameworks is considered at the same level as Ross's rigidity model~\cite{ross2,rossd}, and  we shall extend Theorem~\ref{thm:hendrickson} and Theorem~\ref{thm:jj} to periodic frameworks.
Analogous to Theorem~\ref{thm:hendrickson} there are two types of necessary conditions, a graph connectivity condition (Lemma~\ref{lem:connectivity}) and a redundant rigidity condition (Lemma~\ref{thm:necweak}). 
Our main result (Theorem~\ref{thm:main1}) is that these necessary conditions are also sufficient in $\mathbb{R}^2$, thus giving a first combinatorial characterization of the global rigidity of generic periodic frameworks.
 The proof of this result is inspired by the work in \cite{jjsz,T1}. 
 In particular, it does not require the notion of stress matrices \cite{connelly1982,connelly2005}. Note also that our proof does not rely on periodic global rigidity in $\mathbb{R}^2$ being a generic property, meaning that all generic realizations of a periodic graph in the plane share the same global rigidity properties.

As for the rigidity of periodic frameworks, an extension of Laman's theorem was established in a more general setting by Malestein and Theran~\cite{mt13}, 
where the underlying lattice may deform during a motion of a framework (but it is still subject to being periodic). 
Extending our result to this general setting would be  an important  challenging open problem. 

Important corollaries of our main theorem are toroidal and cylindrical counterparts of Theorem~\ref{thm:jj}.
Here we only give a statement for cylindrical frameworks, but the statement for toroidal frameworks can be  derived in a similar fashion (see Theorem~\ref{thm:torus}).
Consider a straight-line drawing of a graph $G$ on a flat cylinder ${\cal C}$.
We regard it as a bar-joint framework on ${\cal C}$.
Using the metric inherited from its representation as $\mathbb{R}^2/L$ for a fixed one-dimensional lattice $L$, the local/global rigidity is defined.
Ross's theorem~\cite{ross2} for periodic frameworks implies that a generic framework on ${\cal C}$ is rigid if and only if the underlying graph contains a spanning subgraph $H$ with $|E(H)|= 2|V(H)|-2$
such that $|F|\leq 2|V(F)|-2$ for every $F\subseteq E(H)$ and $|F|\leq 2|V(F)|-3$ for every nonempty {\em contractible} $F\subseteq E(H)$,
where $F$ is said to be contractible if every cycle in $F$ is contractible on ${\cal C}$.
\begin{theorem}
\label{thm:cylinder}
A generic framework $(G,p)$ with $|V(G)|\geq 3$ on a flat cylinder ${\cal C}$ is globally rigid if and only if 
it is redundantly rigid on ${\cal C}$, 2-connected, and has no contractible subgraph $H$ with $|V(H)|\geq 3$ and $|B(H)|=2$, 
where $B(H)$ denotes the set of vertices in $H$ incident to some edge not in $H$.
\end{theorem}

The paper is organized as follows.
In Section 2, we define the concept of global rigidity for periodic frameworks with a fixed lattice representation, 
and then establish Hendrickson-type necessary conditions for a generic periodic framework to be globally rigid in $\mathbb{R}^d$ in Section 3. 
In Section 4  we then show that  for $d=2$ the necessary conditions established in Section 3 are also sufficient for generic global rigidity (Theorem~\ref{thm:main1}). 
Section 5 is devoted to the proofs of the combinatorial lemmas stated in Section 4.


\section{Preliminaries}
\label{sec:pre}
\subsection{Periodic graphs}

Let $\Gamma$ be a group isomorphic to $\mathbb{Z}^k$.
In general, a pair $(G,\psi)$ of a directed (multi-)graph $G$ and a map $\psi:E(G)\rightarrow \Gamma$ is called  a {\em $\Gamma$-labeled graph}.

For a given $\Gamma$-labeled graph $(G,\psi)$, one can construct a \emph{$k$-periodic graph} $\tilde{G}$ by setting
$V(\tilde{G})=\{\gamma v_i: v_i\in V(G), \gamma\in \Gamma\}$
and $E(\tilde{G})=\{\{\gamma v_i, \psi(v_iv_j) \gamma v_j\}: (v_i, v_j)\in E(G), \gamma \in \Gamma\}$. $\Gamma$ is called a {\em periodicity} of $\tilde{G}$, which naturally acts on $V(\tilde{G})$ and $E(\tilde{G})$.
This $\tilde{G}$ is also called the {\em covering} of $(G, \psi)$ while $(G, \psi)$ is called the \emph{quotient $\Gamma$-labeled graph} of $\tilde{G}$.
See Figure~\ref{fig:disconn} for two examples of $k$-periodic graphs and corresponding quotient $\Gamma$-labeled graphs.

Although $G$ is directed, its orientation is used only for the reference of the group label, and we are free to change the orientation of
each edge by imposing the property that if an edge has a label $\gamma$ in one direction,
then it has $\gamma^{-1}$ in the other direction. 
More precisely, two edges \(e_1,e_2\) are regarded as {\em identical} if they are parallel with the same direction and the same label, or with the opposite direction and the opposite labels.
Throughout the paper, all  $\Gamma$-labeled
graphs are assumed to be {\em semi-simple}, that is, no identical two edges exist (although
parallel edges may exist). We will assume that \(G\) is loopless. Note that a loop in \(G\) corresponds to an edge orbit in $\tilde{G}$ consiting of edges  that connect vertices in the same orbit in $\tilde{G}$. In our $L$-periodic (global) rigidity model defined in Sections~\ref{subsec:global} and \ref{subsec:local} such edge constraints are redundant thus do not affect the (global) rigidity of the framework.

We define a \emph{walk} as an alternating sequence $v_1,e_1,v_2\ldots, e_k,v_{k+1}$  of vertices and edges such that $v_i$ and $v_{i+1}$ are the end vertices of $e_i$. For a  closed walk $C=v_1,e_1,v_2\ldots, e_k,v_1$ in $(G,\psi)$, let $\psi(C)=\prod_{i=1}^k\psi(e_i)^{\textrm{sign}(e_i)}$, where $\textrm{sign}(e_i)=1$ if $e_i$ has forward direction in $C$, and $\textrm{sign}(e_i)=-1$ otherwise.
For a subgraph $H$ of $G$ define
$\Gamma_H$ as the subgroup of $\Gamma$ generated by the elements $\psi(C)$, where $C$ ranges  over all closed walks in $H$.
The {\em rank} of $H$ is defined to be the rank of $\Gamma_H$.
Note that the rank of $G$ may be less than the rank of $\Gamma$, in which case the covering graph $\tilde{G}$ contains an infinite number of connected components (see Figure~\ref{fig:disconn}(a)).

We say that an edge set $F$ is {\em balanced} if the subgraph induced by $F$ has rank zero, i.e., $\psi(C)=\textrm{id}$ for every closed walk $C$ in $F$.

One useful tool to compute the rank of a subgraph is the {\em switching} operation.
A {\em switching} at $v\in V(G)$  by $\gamma\in \Gamma$ changes $\psi$ to $\psi'$
defined by $\psi'(e)=\gamma \psi(e)$ if $e$ is directed from $v$,
$\psi'(e)=\gamma^{-1} \psi(e)$ if $e$ is directed to $v$,
and $\psi'(e)=\psi(e)$ otherwise. A switching operation preserves the rank of any subgraph (see \cite{jkt} for details).

Note that the quotient $\Gamma$-labeled graph of a $k$-periodic graph is not unique in general. It is not difficult to see that $(G=(V,E), \psi)$ and $(G'=(V,E'), \psi')$ define the same covering graph $\tilde{G}$ if and only if $(G, \psi)$ can be transformed into $(G', \psi')$ via edge reversions (as described above) and switching operations.

\begin{figure}[h!]

\begin{tikzpicture}[scale=0.8]
\tikzstyle{every node}=[circle, draw=black, fill=white, inner sep=0pt, minimum width=4pt];
\hspace{-0.8cm}
    \begin{scope}
        \myGlobalTransformation{0}{0};
        \foreach \x in {1,3,5,7} {
            \foreach \y in {1,3,5,7} {
                \node (a\x\y) at (\x-0.3,\y) {};
                \node (b\x\y) at (\x+0.5,\y+0.3) {};
                \node (c\x\y) at (\x+0.6,\y-0.3) {};
                {
                    \pgftransformreset
										\draw[black, thick] (a\x\y) -- (b\x\y);
                    \draw[black, thick] (a\x\y) -- (c\x\y);
                }
            }
        }
    \end{scope}
    \begin{scope}
        \myGlobalTransformation{0}{0};
        \foreach \x in {-1,1,3,5,7,9} {
            \foreach \y in {-1,1,3,5,7,9} {
                \node (a\x\y) at (\x-0.3,\y) {};
                \node (b\x\y) at (\x+0.5,\y+0.3) {};
                \node (c\x\y) at (\x+0.6,\y-0.3) {};
                {
									\ifnum\x<9
									\ifnum\y<9
										\draw[black, thick] (b\x\y) -- ++(2.1,1.4); 
										\draw[black, thick] (b\x\y) -- ++(1.2,1.7); 
                  \fi
									\fi
                }
            }
        }
    \end{scope}
    \gridThreeD{0}{0}{black,dotted};
  \node [rectangle, draw=white, fill=white] (a) at (4.8,-1.2) {(a)};
\hspace{0.8cm}
\begin{scope}
        \myGlobalTransformation{0}{0};
                \node (a2) at (10.2,3) {};
                \node (b2) at (11.5,5.9) {};
                \node (c2) at (12.8,1.5) {};
								\draw[->, black, thick] (a2) .. controls (10.8,3.8) .. (b2) node[draw=white, midway]{\small{$(0,0)$}};
								\draw[->, black, thick] (a2) -- (c2) node[draw=white, midway]{\small{$(0,0)$}};
								\draw[<-, black, thick] (a2) .. controls (9.8,5.8) .. (b2) node[draw=white, midway]{\small{$(1,1)$}};
								\draw[->, black, thick] (b2) -- (c2) node[draw=white, midway]{\small{$(1,1)$}};
\end{scope}
  \node [rectangle, draw=white, fill=white] (b) at (12.8,-1.2) {(b)};

\end{tikzpicture}

\begin{tikzpicture}[scale=0.8]
\tikzstyle{every node}=[circle, draw=black, fill=white, inner sep=0pt, minimum width=4pt];
\hspace{-0.8cm}
    \begin{scope}
        \myGlobalTransformation{0}{0};
        \foreach \x in {1,3,5,7} {
            \foreach \y in {1,3,5,7} {
                \node (a\x\y) at (\x-0.3,\y) {};
                \node (b\x\y) at (\x+0.5,\y+0.3) {};
                \node (c\x\y) at (\x+0.6,\y-0.3) {};
                {
                    \pgftransformreset
										\draw[black, thick] (a\x\y) -- (b\x\y);
                    \draw[black, thick] (a\x\y) -- (c\x\y);
                }
            }
        }
    \end{scope}
    \begin{scope}
        \myGlobalTransformation{0}{0};
        \foreach \x in {-1,1,3,5,7,9} {
            \foreach \y in {-1,1,3,5,7,9} {
                \node (a\x\y) at (\x-0.3,\y) {};
                \node (b\x\y) at (\x+0.5,\y+0.3) {};
                \node (c\x\y) at (\x+0.6,\y-0.3) {};
                {
									\ifnum\x<9
									\ifnum\y<7
										\draw[black, thick] (b\x\y) -- ++(2.1,1.4); 
										\draw[black, thick] (b\x\y) -- ++(-0.8,3.7);
									\fi
									\fi
									\ifnum\x<9
									\ifnum\y=7
										\draw[black, thick] (b\x\y) -- ++(2.1,1.4); 
										\draw[black, thick] (b\x\y) -- ++(-0.4,1.85);
                  \fi
									\fi
									\ifnum\x<9
									\ifnum\y=1
										\draw[black, thick] (b\x\y) -- ++(2.1,1.4); 
										\draw[black, thick] (a\x\y) -- ++(0.4,-1.85);
                  \fi
									\fi
                }
            }
        }
    \end{scope}
		\begin{scope}
        \myGlobalTransformation{0}{0};
				    \foreach \x in {1,5} {
            \foreach \y in {1,5} {
                \node[fill=black] (a\x\y) at (\x-0.3,\y) {};
                \node[fill=black] (b\x\y) at (\x+0.5,\y+0.3) {};
                \node[fill=black] (c\x\y) at (\x+0.6,\y-0.3) {};
								}}
				    \foreach \x in {3,7} {
            \foreach \y in {3,7} {
                \node[fill=black] (a\x\y) at (\x-0.3,\y) {};
                \node[fill=black] (b\x\y) at (\x+0.5,\y+0.3) {};
                \node[fill=black] (c\x\y) at (\x+0.6,\y-0.3) {};
								}}
    \end{scope}
    \gridThreeD{0}{0}{black,dotted};
  \node [rectangle, draw=white, fill=white] (c) at (4.8,-1.2) {(c)};
\hspace{0.8cm}
\begin{scope}
        \myGlobalTransformation{0}{0};
                \node (a2) at (10.2,3) {};
                \node (b2) at (11.5,5.9) {};
                \node (c2) at (12.8,1.5) {};
								\draw[->, black, thick] (a2) .. controls (10.8,3.8) .. (b2) node[draw=white, midway]{\small{$(0,0)$}};
								\draw[->, black, thick] (a2) -- (c2) node[draw=white, midway]{\small{$(0,0)$}};
								\draw[<-, black, thick] (a2) .. controls (9.8,5.8) .. (b2) node[draw=white, midway]{\small{$(0,2)$}};
								\draw[->, black, thick] (b2) -- (c2) node[draw=white, midway]{\small{$(1,1)$}};
\end{scope}
  \node [rectangle, draw=white, fill=white] (d) at (12.8,-1.2) {(d)};
\end{tikzpicture}

\vspace{-0.3cm}
\caption{
The framework in (a)  has infinitely many connected components, and its quotient $\mathbb{Z}^2$-labeled graph shown in (b) is of rank one.
The framework in (c) has  two connected components, and its  quotient $\mathbb{Z}^2$-labeled graph shown in (d) is of rank two. The vertices of the two components are depicted with two distinct colors in (c).
We remark that neither of these frameworks is \(L\)-periodically globally rigid, by Lemma~\ref{lem:connectivity}.} 
\label{fig:disconn}
\end{figure}
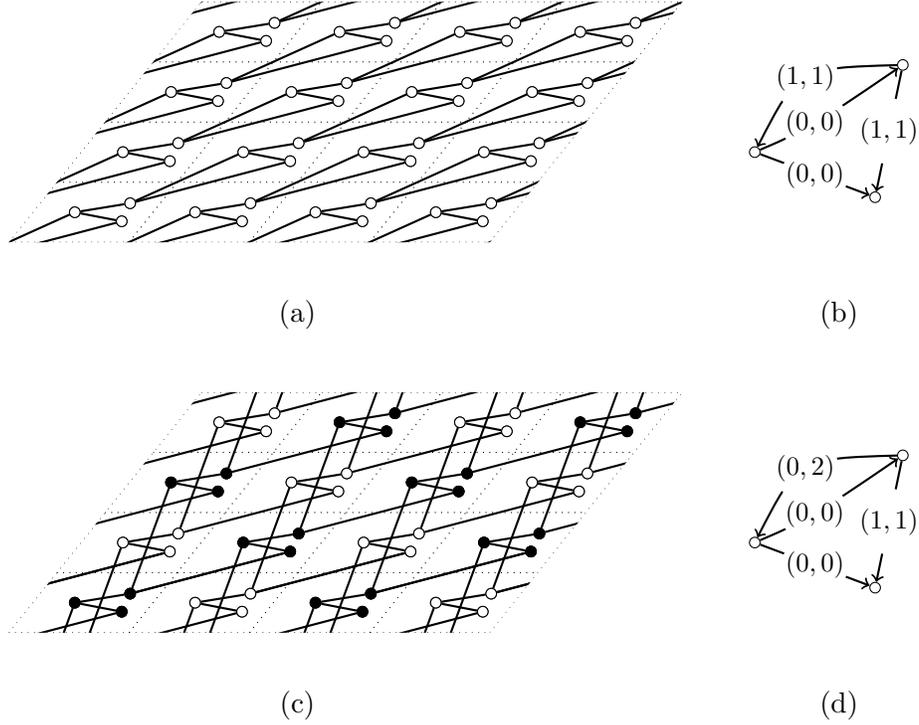

\subsection{Periodic frameworks}
In the context of graph rigidity, a pair $(G,p)$ of a graph $G=(V,E)$ and a map
$p:V\rightarrow \mathbb{R}^d$ is called a {\em (bar-joint) framework} in $\mathbb{R}^d$.
A periodic framework is a special type of infinite framework defined as follows.

Let $\tilde{G}$ be a $k$-periodic graph with periodicity $\Gamma$,
and let  $L:\Gamma \rightarrow \mathbb{R}^d$ be a nonsingular homomorphism with $k\leq d$,
where $L$ is said to be nonsingular if $L(\Gamma)$ has rank $k$.
A pair $(\tilde{G},\tilde{p})$ of $\tilde{G}$ and $\tilde{p}:\tilde{V}\rightarrow \mathbb{R}^d$ is said to be an
{\em $L$-periodic framework} in $\mathbb{R}^d$ if  \begin{equation}
\label{eq:periodic_func}
\tilde{p}(v)+L(\gamma)=\tilde{p}(\gamma v) \qquad \text{for all } \gamma\in \Gamma \text{ and all } v\in \tilde{V}.
\end{equation}
We also say that a pair $(\tilde{G}, \tilde{p})$ is {\em $k$-periodic} in $\mathbb{R}^d$ if it is $L$-periodic for some nonsingular homomorphism $L:\Gamma\rightarrow \mathbb{R}^d$.
Note that the rank $k$ of the periodicity may be smaller than $d$.
For instance, if $k=1$ and $d=2$, then $(\tilde{G}, \tilde{p})$ is an infinite strip framework in the plane. Any connected component of the framework in Figure~\ref{fig:disconn}(a), for example, is such a strip framework (with $\Gamma=\langle(1,1)\rangle$).

An $L$-periodic framework $(\tilde{G}, \tilde{p})$ is {\em generic} if the set of coordinates is algebraically independent over the rationals modulo the ideal generated by the equations (\ref{eq:periodic_func}). For simplicity of description, throughout the paper, we shall assume that $L$ is a rational-valued function\footnote{This assumption is not essential. Let \(\gamma_1,\dots,\gamma_k\) be a set of generators of \(\Gamma\) and let \(H\) be the set of coordinates of the vectors \(L(\gamma_i)\) for \(1\leq i\leq k\). Then one can apply the subsequent arguments even if $L$ is not rational by replacing $\mathbb{Q}$ with the field extension \(\mathbb{Q}(H)\). If \(L\) is not rational-valued, then we say that $(\tilde{G}, \tilde{p})$ is {\em generic} if the set of coordinates is algebraically independent over \(\mathbb{Q}(H)\) modulo the ideal generated by the equations (\ref{eq:periodic_func}).},
i.e., $L:\Gamma \rightarrow \mathbb{Q}^d$.

Let $(\tilde{G}, \tilde{p})$ be an $L$-periodic framework and  let $(G,\psi)$ be the quotient $\Gamma$-labeled graph of $\tilde{G}$.
Following the convention that $V(G)$ is identified with the set $\{v_1,\dots, v_n\}$ of representative vertices,
one can define the {\em quotient $\Gamma$-labeled framework} as the triple $(G,\psi, p)$
with $p: V(G)\ni v_i\mapsto \tilde{p}(v_i)\in \mathbb{R}^d$.
In general, a {\em $\Gamma$-labeled framework } is defined to be a triple $(G,\psi, p)$ of a finite $\Gamma$-labeled graph $(G,\psi)$ and a map $p:V(G)\rightarrow \mathbb{R}^d$.
The {\em covering} of $(G,\psi,p)$ is a $k$-periodic framework $(\tilde{G}, \tilde{p})$,
where $\tilde{G}$ is the covering of $G$ and $\tilde{p}$ is uniquely determined from $p$ by (\ref{eq:periodic_func}).

We say that a $\Gamma$-labeled framework $(G,\psi,p)$ is {\em generic} if  the set of coordinates in \(p\) is algebraically independent over the rationals.
Note that an $L$-periodic framework $(\tilde{G}, \tilde{p})$ is generic if and only if the quotient $(G,\psi,p)$ of $(\tilde{G}, \tilde{p})$ is generic.

\subsection{Rigidity and global rigidity} \label{subsec:global}
Let $G=(V,E)$ be a graph.
Two frameworks $(G,p)$ and $(G,q)$ in $\mathbb{R}^d$ are said to be \emph{equivalent} if
\begin{equation*}
\label{eq:periodic_equiv}
\left\| p(u)-p(v)\right\|=\left\| q(u)-q(v)\right\|\qquad \text{for all } uv\in E.
\end{equation*}
They are \emph{congruent} if
\begin{equation*}
\label{eq:periodic_congr}
\left\| p(u)-p(v)\right\|=\left\| q(u)- q(v)\right\|\qquad \text{for all } u,v\in V.
\end{equation*}
A framework $(G,p)$ is called {\em globally rigid} if every framework $(G,q)$ in $\mathbb{R}^d$ which is equivalent to $(G,p)$ is also congruent to $(G,p)$.

We may define the corresponding periodicity-constrained concept as follows.
An $L$-periodic framework $(\tilde{G}, \tilde{p})$  in $\mathbb{R}^d$ is {\em $L$-periodically globally rigid} if
every $L$-periodic framework in $\mathbb{R}^d$ which is equivalent to $(\tilde{G}, \tilde{p})$ is also congruent to $(\tilde{G}, \tilde{p})$.
Note that if the rank of the periodicity is equal to zero, then $L$-periodic global rigidity coincides with the global rigidity of finite frameworks.

A key notion to analyze  $L$-periodic global rigidity is  $L$-periodic {\em  rigidity}.
By the periodicity constraint (\ref{eq:periodic_func}), the space of all $L$-periodic frameworks in $\mathbb{R}^d$ for a given graph $\tilde{G}$ can be identified with Euclidean space $\mathbb{R}^{dn}$ so that the topology is defined.
A framework $(\tilde{G},\tilde{p})$ is called {\em $L$-periodically rigid} if there is an open neighborhood $N$ of $\tilde{p}$ in which  every $L$-periodic framework $(\tilde{G},\tilde{q})$ which is equivalent to $(\tilde{G},\tilde{p})$ is also congruent to $(\tilde{G},\tilde{p})$.

\subsection{Characterizing $L$-periodic rigidity} \label{subsec:local}
A key tool to analyze the local or the global rigidity of finite frameworks is the length-squared function and its Jacobian, called the rigidity matrix.
One can follow the same strategy to analyze local or global periodic rigidity.

For a $\Gamma$-labeled graph $(G,\psi)$ and $L:\Gamma\rightarrow \mathbb{R}^d$, define
$f_{G,L}:\mathbb{R}^{d|V(G)|}\rightarrow \mathbb{R}^{|E(G)|}$ by
\[
f_{G,L}(p)=(\dots,\left\|p(v_i)-(p(v_j)+L(\psi(v_iv_j)))\right\|^2,\dots) \qquad (p\in \mathbb{R}^{d|V(G)|}).
\]
For a finite set $V$, the {\em complete $\Gamma$-labeled graph} $K(V,\Gamma)$ on $V$ is defined
to be the graph on $V$ with the edge set $\{(u, \gamma v): u,v \in V, \gamma\in \Gamma\}$.
We simply denote $f_{K(V,\Gamma), L}$ by $f_{V,L}$.

By (\ref{eq:periodic_func}) we have the following fundamental fact.
\begin{proposition}
\label{prop:trivial}
Let $(\tilde{G}, \tilde{p})$ be an $L$-periodic framework and let $(G=(V,E), \psi, p)$ be a quotient $\Gamma$-labeled framework of $(\tilde{G}, \tilde{p})$. Then $(\tilde{G}, \tilde{p})$ is $L$-periodically globally (resp.~locally) rigid
if and only if for every $q\in \mathbb{R}^{d|V|}$ (resp.~for every $q$ in an open neighborhood of $p$ in $\mathbb{R}^{d|V|}$), $f_{G,L}(p)=f_{G,L}(q)$ implies $f_{V,L}(p)=f_{V, L}(q)$.
\end{proposition}
In view of this proposition, we say that a $\Gamma$-labeled framework $(G,\psi,p)$ is {\em $L$-periodically globally (or, locally) rigid} if for every $q\in \mathbb{R}^{d|V|}$ (resp. for every $q$ in an open neighborhood of $p$ in $\mathbb{R}^{d|V|}$), $f_{G,L}(p)=f_{G,L}(q)$ implies $f_{V,L}(p)=f_{V, L}(q)$,
and we may focus on characterizing the $L$-periodic global (or, local) rigidity of $\Gamma$-labeled frameworks.

For $X\subset \mathbb{R}^d$, an isometry $h:\mathbb{R}^d\rightarrow \mathbb{R}^d$ of $\mathbb{R}^d$ is said to be {\em $X$-invariant} if $h(x)-h(0)=x$ for every $x\in X$.
The following is the first fundamental fact for analyzing periodic rigidity.
\begin{proposition}
\label{prop:isometry}
Let $V$ be a finite set, $p, q:V\rightarrow \mathbb{R}^d$ two maps, and let $L:\Gamma\rightarrow \mathbb{R}^d$ be nonsingular.
Suppose that $p(V)$ affinely spans $\mathbb{R}^d$.
Then $f_{V,L}(p)=f_{V,L}(q)$ holds if and only if $q$ can be written as $q=h\circ p$ for some $L(\Gamma)$-invariant isometry $h$ of $\mathbb{R}^d$.
\end{proposition}
\begin{proof}
To see the necessity, assume that $f_{V,L}(p)=f_{V,L}(q)$.
Then
$\|p(u)-p(w)\|=\|q(u)-q(w)\|$ for every pair of elements $u,w\in V$.
Since $p(V)$ affinely spans the whole space,
this implies that there is a unique isometry $h:\mathbb{R}^d\rightarrow \mathbb{R}^d$
such that $q(u)=h(p(u))$ for every $u\in V$.
In other words, there is an orthogonal matrix $S$ such that $q(u)-q(v)=S(p(u)-p(v))$ for every $u, v\in V$.
By $f_{V,L}(p)=f_{V,L}(q)$, we have
$\|p(v)-(p(u)+L(\gamma))\|=\|q(v)-(q(u)+L(\gamma))\|$ for every $u, v\in V$ and every $\gamma\in \Gamma$,
which implies  $\langle (I-S)(p(u)-p(v)), L(\gamma)\rangle=0$,
and hence $\langle p(u)-p(v), (I-S^{\top})L(\gamma)\rangle=0$.
Since $p(V)$ affinely spans the whole space, $(I-S^{\top})L(\gamma)=0$ for every $\gamma$.
In other words, $S$ fixes each element in $L(\Gamma)$ as required.

Conversely, suppose that  $q$ is  represented as
$q(v)=S p(v)+t$ for some  $t\in \mathbb{R}^d$ and some orthogonal matrix $S$ that fixes each element in $L(\Gamma)$.
Then we have $\|q(v)-(q(u)+L(\gamma))\|=\|S(p(v)-(p(u)+L(\gamma)))\|=\|p(v)-(p(u)+L(\gamma))\|$ for every $u, v\in V$ and every $\gamma\in \Gamma$, implying $f_{V,L}(q)=f_{V,L}(p)$.
\end{proof}

The following algebraic characterization is the periodic version of one of the fundamental facts in rigidity theory,
and is known in the periodic case even if $k=d$~\cite{rossd}.
\begin{proposition}
\label{prop:rigidity_matrix}
Let $(G,\psi, p)$ be a  generic $\Gamma$-labeled framework in $\mathbb{R}^d$
with $|V(G)|\geq d+1$ and rank $k$ periodicity $\Gamma$, and let $L:\Gamma\rightarrow \mathbb{R}^d$ be nonsingular.
Then $(G,\psi,p)$ is $L$-periodically rigid if and only if
\[
\rank df_{G,L}|_p=d|V(G)|-d-{d-k\choose 2},
\]
where $df_{G,L}|_p$ denotes the Jacobian of $f_{G,L}$ at $p$.
\end{proposition}
\begin{proof}
Since the rank of $\Gamma$ is $k$, the set of $L(\Gamma)$-invariant isometries forms a $(d+{d-k\choose 2})$-dimensional manifold. Hence $\rank df_{V,L}|_p=d|V(G)|-d-{d-k\choose 2}$.
By the standard argument using the inverse function theorem,
it follows that $(G,\psi,p)$ is $L$-periodically rigid if and only if $\rank df_{V,L}|_p=\rank df_{G,L}|_p$,
implying the statement.
\end{proof}
 For $d=2$ Ross~\cite{ross2} gave a combinatorial characterization of the rank of $df_{G,L}|_p$ for generic $(G,\psi,p)$, which implies the following. (Her statement is only for $k=2$, but the proof can easily be adapted to the case when $k=1$.)

 \begin{theorem} [Ross~\cite{ross2}]
 \label{thm:ross}
 Let $(G,\psi, p)$ be a generic $\Gamma$-labeled framework in $\mathbb{R}^2$  with rank $k\geq 1$ periodicity $\Gamma$ and let $L:\Gamma\rightarrow \mathbb{R}^2$ be nonsingular.
 Then $(G,\psi, p)$ is $L$-periodically rigid if and only if
 $(G,\psi)$ contains a spanning subgraph $(H,\psi_H)$ satisfying the following count conditions:
 \begin{itemize}
 \item $|E(H)|=2|V(G)|-2$;
 \item  $|F|\leq 2|V(F)|-3$ for every nonempty balanced $F\subseteq E(H)$;
 \item $|F|\leq 2|V(F)|-2$ for every nonempty $F\subseteq E(H)$.
 \end{itemize}
 \end{theorem}
    Note that if $k=0$, then an $L$-periodic framework is simply a finite framework,
    and  generic rigid frameworks in $\mathbb{R}^2$ are characterized by the celebrated Laman theorem \cite{Lamanbib}.

\section{Necessary Conditions}
\label{sec:nec}
In this section we provide necessary conditions for $L$-periodic global rigidity.
As in the finite case, there are two types of conditions, a connectivity condition and a redundant rigidity condition.
These two conditions are stated in Lemma~\ref{lem:connectivity} and Lemma~\ref{thm:necweak}, respectively.

\subsection{Necessary connectivity conditions}
Let $(G,\psi)$ be a $\Gamma$-labeled graph with $\Gamma$ having rank $k$.
For a subgraph $H$ of $G$, the {\em boundary} $B(H)$ is defined to be the set of vertices in $H$ incident to some edge in $E(G)\setminus E(H)$, 
and the {\em interior} $I(H)$ is defined by $I(H)=V(H)\setminus B(H)$.
A subgraph $H$ is said to be an {\em $(s,t)$-block} if 
the rank of $H$ is $s$, $|B(H)|=t$, and $I(H)\neq \emptyset$.

\begin{lemma}
\label{lem:connectivity}
Let $(G, \psi, p)$ be a generic $\Gamma$-labeled framework in $\mathbb{R}^d$ with $|V(G)|\geq 2$,
rank $k$ periodicity $\Gamma$, and nonsingular $L:\Gamma\rightarrow \mathbb{R}^d$.
If $(G, \psi, p)$ is $L$-periodically globally rigid, 
then $G$ contains no $(s,t)$-block $H$ with $s+t\leq d$
such that $V(H)\neq V(G)$ or $s<k$.
\end{lemma}
\begin{proof}
We first remark that the following holds for any $\Gamma$-labeled graph $(H, \psi_H)$ and any $v\in V(H)$:
\begin{equation}
\label{eq:separation0}
\begin{split}
&\text{If a hyperplane ${\cal H}$ of $\mathbb{R}^d$ contains $\{p(v)+L(\gamma): \gamma\in \Gamma_H\}$, then} \\
&\text{$f_{H,L}(g\circ p)=f_{H,L}(p)$ holds for the reflection  $g:\mathbb{R}^d\rightarrow \mathbb{R}^d$  with respect to ${\cal H}$.}
\end{split}
\end{equation}
This follows from Proposition~\ref{prop:isometry} by noting that the reflection $g$ is $L(\Gamma_H)$-invariant.

Suppose that $G$ has an $(s,t)$-block $H$ satisfying the property of the statement. 
If $t=0$ and $V(H)\neq V(G)$ (i.e., $G$ is disconnected), then by translating $p(V(H))$ we obtain $q$ with $f_{G,L}(p)= f_{G,L}(q)$ and $f_{V,L}(p)\neq f_{V,L}(q)$, contradicting the global rigidity of $(G,\psi,p)$. 
Thus we have
\begin{equation}
\label{eq:cut1}
t>0 \text{ or } V(H)=V(G).
\end{equation}

Define $B'$ by $B'=B(H)$ if $t>0$ and otherwise $B'=\{x\}$ by picking any vertex $x\in V(G)$, 
and consider the set of points $P=\{p(v)+L(\gamma): v\in B', \gamma\in \Gamma_{H}\}$.
We have the following:
\begin{equation}
\label{eq:cut2}
\text{The affine span $\aff P$ of $P$ is a proper subspace of $\mathbb{R}^d$.}
\end{equation}
Indeed, if $t>0$, then $B'=B(H)$ and $\aff P$ has dimension at most $s+t-1$, which is less than $d$ by the lemma assumption.
On the other hand,  if $t=0$, then $B'=\{x\}$ and $\aff P$ has dimension $s$, which is less than $d$ 
by (\ref{eq:cut1}) and the lemma assumption. Thus (\ref{eq:cut2}) follows.

By (\ref{eq:cut2}) we can take a hyperplane $\cal H$ that contains $\aff P$.
Since $p$ is generic, such a hyperplane can be taken such that $\cal H$ contains no point in $\{p(v): v\in V(G)\setminus B'\}$.
Let $g:\mathbb{R}^d\rightarrow \mathbb{R}^d$ be the reflection of $\mathbb{R}^d$ with respect to ${\cal H}$,
and we define $q:V\rightarrow \mathbb{R}^d$ by $q(v)=g(p(v))$ if $v\in V(H)$, and $q(v)=p(v)$ otherwise.

We first show $f_{G,L}(p)=f_{G,L}(q)$.
Take any edge $e=uv$ in $G$.
If $e\notin E(H)$, then $p(u)=q(u)$ and $p(v)=q(v)$ (in particular, $g$ is identity on $p(B(H))$),
implying $\|p(u)-(p(v)+L(\psi(e)))\|=\|q(u)-(q(v)+L(\psi(e)))\|$.
Otherwise, (\ref{eq:separation0}) implies $\|p(u)-(p(v)+L(\psi(e)))\|=\|q(u)-(q(v)+L(\psi(e)))\|$.
Thus $f_{G,L}(p)=f_{G,L}(q)$ follows.

To derive a contradiction we shall show $f_{V,L}(p)\neq f_{V,L}(q)$ by splitting the proof into two cases depending on whether $V(H)=V(G)$ or not.

Suppose that $V(G)\neq V(H)$.
By the definition of an $(s,t)$-block, we can take a \(v\in I(H)\).
Then  $f_{V,L}(p)\neq f_{V,L}(q)$ holds
since  $\|p(v)-p(u)\|\neq \|q(v)-q(u)\|$ for any $u\in V(G)\setminus V(H)$.

Suppose that $V(H)=V(G)$.
Let $x\in B'$ and $u\in I(H)\setminus B'$. (As $x$ was chosen arbitrary from $V(G)$, we may suppose 
$I(H)\setminus B'\neq \emptyset$.)
By the lemma assumption,  $s<k$ holds, and hence we can take $\gamma^*\in \Gamma$ that is not spanned by $\Gamma_H$.
As $\gamma^*$ is not spanned by $\Gamma_H$,  
we could take the above hyperplane ${\cal H}$ such that $p(x)-L(\gamma^*)$ is outside of ${\cal H}$.
Note that $p(u)\neq q(u)$ as $p(u)\notin {\cal H}$ (by $u\in I(H)\setminus B'$).
Hence the set of points equidistant from $p(u)$ and $q(u)$ is ${\cal H}$,
which in turn implies that 
the set of points equidistant from $p(u)+L(\gamma^*)$ and $q(u)+L(\gamma^*)$ is ${\cal H}+L(\gamma^*)$.
Therefore, as $p(x)\notin {\cal H}+L(\gamma^*)$ but $p(x)\in {\cal H}$, 
$\|p(x)-(p(u)+L(\gamma^*))\|\neq \|p(x)-(q(u)+L(\gamma^*))\|=\|q(x)-(q(u)+L(\gamma^*))\|$.
This implies $f_{V,L}(p)\neq f_{V,L}(q)$.
\end{proof}

\begin{figure}[h!]

\begin{tikzpicture}[scale=0.8]
\tikzstyle{every node}=[circle, draw=black, fill=white, inner sep=0pt, minimum width=4pt];
\hspace{-1.4cm}
    \begin{scope}
        \myGlobalTransformation{0}{0};
        \foreach \x in {1,3,5,7} {
            \foreach \y in {1,3,5,7} {
                \node (b\x\y) at (\x+0.5,\y+0.3) {};
                \node (c\x\y) at (\x+0.6,\y-0.3) {};
								\node (d\x\y) at (\x-0.65,\y-0.5) {};
								\node (e\x\y) at (\x-0.2,\y+0.2) {};
                {
                    \pgftransformreset
										\draw[black, thick, dashed] (d\x\y) -- (e\x\y);
                    \draw[black, thick] (b\x\y) -- (c\x\y);
										\draw[black, thick, dashed] (d\x\y) -- (c\x\y);
                    \draw[black, thick, dashed] (c\x\y) -- (e\x\y);
                }
            }
        }
    \end{scope}
    \begin{scope}
        \myGlobalTransformation{0}{0};
        \foreach \x in {-1,1,3,5,7,9} {
            \foreach \y in {-1,1,3,5,7,9} {
                \node (b\x\y) at (\x+0.5,\y+0.3) {};
                \node[fill=black] (c\x\y) at (\x+0.6,\y-0.3) {};
								\node (d\x\y) at (\x-0.65,\y-0.5) {};
								\node (e\x\y) at (\x-0.2,\y+0.2) {};
                {
									\ifnum\x<9
									\ifnum\y<9
										\draw[black, thick] (b\x\y) -- ++(2.1,1.4); 
										\draw[black, thick] (b\x\y) -- ++(2.1,-0.6); 
										\draw[black, thick, dashed] (e\x\y) -- ++(-0.45,1.3); 
                  \fi
									\fi
                }
            }
        }
    \end{scope}
    \begin{scope}
        \myGlobalTransformation{0}{0};
        \foreach \x in {-1,1,3,5,7,9} {
										\draw[black, very thin] (c\x-1) -- (c\x9);
        }
    \end{scope}
    \gridThreeD{0}{0}{black,dotted};
  \node [rectangle, draw=white, fill=white] (a) at (4.8,-1.2) {(a)};
\hspace{1cm}
\begin{scope}
        \myGlobalTransformation{0}{0};
                \node (a2) at (10.2,3) {};
                \node (b2) at (11.5,5.9) {};
                \node[fill=black] (c2) at (12.8,1.5) {};
                \node (d2) at (16,0) {};
								\draw[->, black, dashed, thick] (a2) .. controls (10.8,3.8) .. (b2) node[draw=white, midway]{\small{$(0,0)$}};
								\draw[<-, black, dashed, thick] (a2) -- (c2) node[draw=white, midway]{\small{$(0,0)$}};
								\draw[<-, black, dashed, thick] (a2) .. controls (9.8,5.8) .. (b2) node[draw=white, midway]{\small{$(0,1)$}};
								\draw[<-, black, dashed, thick] (b2) -- (c2) node[draw=white, midway]{\small{$(0,0)$}};
								\draw[<-, black, thick] (d2) -- (c2) node[draw=white, midway]{\small{$(0,0)$}};
								\draw[->, black, thick] (d2) .. controls (13.8,3) .. (c2) node[draw=white, midway]{\small{$(1,0)$}};
								\draw[->, black, thick] (d2) .. controls (14.8,-1) .. (c2) node[draw=white, midway]{\small{$(1,1)$}};
\end{scope}
  \node [rectangle, draw=white, fill=white] (b) at (13.2,-1.2) {(b)};
\end{tikzpicture}
\begin{tikzpicture}[scale=0.8]
\tikzstyle{every node}=[circle, draw=black, fill=white, inner sep=0pt, minimum width=4pt];
\hspace{-1.4cm}
    \begin{scope}
        \myGlobalTransformation{0}{0};
        \foreach \x in {1,3,5,7} {
            \foreach \y in {1,3,5,7} {
                \node (b\x\y) at (\x+0.5,\y+0.7) {};
                \node (c\x\y) at (\x+0.6,\y-0.3) {};
								\node (d\x\y) at (\x-0.65,\y-0.5) {};
								\node (e\x\y) at (\x-0.2,\y+0.4) {};
                    \draw[black, very thin] (b\x\y) -- ++ (-0.15,1.5);
                    \draw[black, very thin] (c\x\y) -- ++ (0.15,-1.5);
                {
                    \pgftransformreset
										\draw[black, thick, dashed] (d\x\y) -- (e\x\y);
                    \draw[white, very thick] (b\x\y) -- (c\x\y);
                    \draw[black, thick, dashed] (b\x\y) -- (c\x\y);
										\draw[black, thick, dashed] (d\x\y) -- (c\x\y);
                    \draw[black, thick, dashed] (b\x\y) -- (e\x\y);
                }
            }
        }
    \end{scope}
    \begin{scope}
        \myGlobalTransformation{0}{0};
        \foreach \x in {-1,1,3,5,7,9} {
            \foreach \y in {-1,1,3,5,7,9} {
                \node[fill=black] (b\x\y) at (\x+0.5,\y+0.7) {};
                \node[fill=black] (c\x\y) at (\x+0.6,\y-0.3) {};
								\node (d\x\y) at (\x-0.65,\y-0.5) {};
								\node (e\x\y) at (\x-0.2,\y+0.4) {};
                {
									\ifnum\x<9
									\ifnum\y<9
										\draw[black, thick] (b\x\y) -- ++(2.1,-1); 
										\draw[black, thick] (b\x\y) -- ++(2.1,1); 
                  \fi
									\fi
                }
            }
        }
    \end{scope}
    \gridThreeD{0}{0}{black,dotted};
  \node [rectangle, draw=white, fill=white] (a) at (4.8,-1.2) {(c)};
\hspace{1cm}
\begin{scope}
        \myGlobalTransformation{0}{0};
                \node (a2) at (9.2,6) {};
                \node (b2) at (12,6.9) {};
                \node[fill=black] (c2) at (11.8,1.5) {};
                \node[fill=black] (d2) at (14,3.1) {};
								\draw[->, black, dashed, thick] (a2) -- (b2) node[draw=white, midway]{\small{$(0,0)$}};
								\draw[<-, black, dashed, thick] (a2) -- (c2) node[draw=white, midway]{\small{$(0,0)$}};
								\draw[<-, black, dashed, thick] (b2) -- (d2) node[draw=white, midway]{\small{$(0,0)$}};
								\draw[<-, black, dashed, thick] (d2) -- (c2) node[draw=white, midway]{\small{$(0,0)$}};
								\draw[->, black, thick] (d2) .. controls (11.5,4.5) .. (c2) node[draw=white, midway]{\small{$(1,0)$}};
								\draw[->, black, thick] (d2) .. controls (14,0) .. (c2) node[draw=white, midway]{\small{$(1,1)$}};
\end{scope}
  \node [rectangle, draw=white, fill=white] (b) at (13.2,-1.2) {(d)};

\end{tikzpicture}

\vspace{-0.3cm}
\caption{Examples of  frameworks which are not globally \(L\)-periodically rigid, illustrating the proof of Lemma~\ref{lem:connectivity} for a rank 2 group  \(\Gamma=\mathbb{Z}^2\). 
}
\label{fig:sep}
\end{figure}
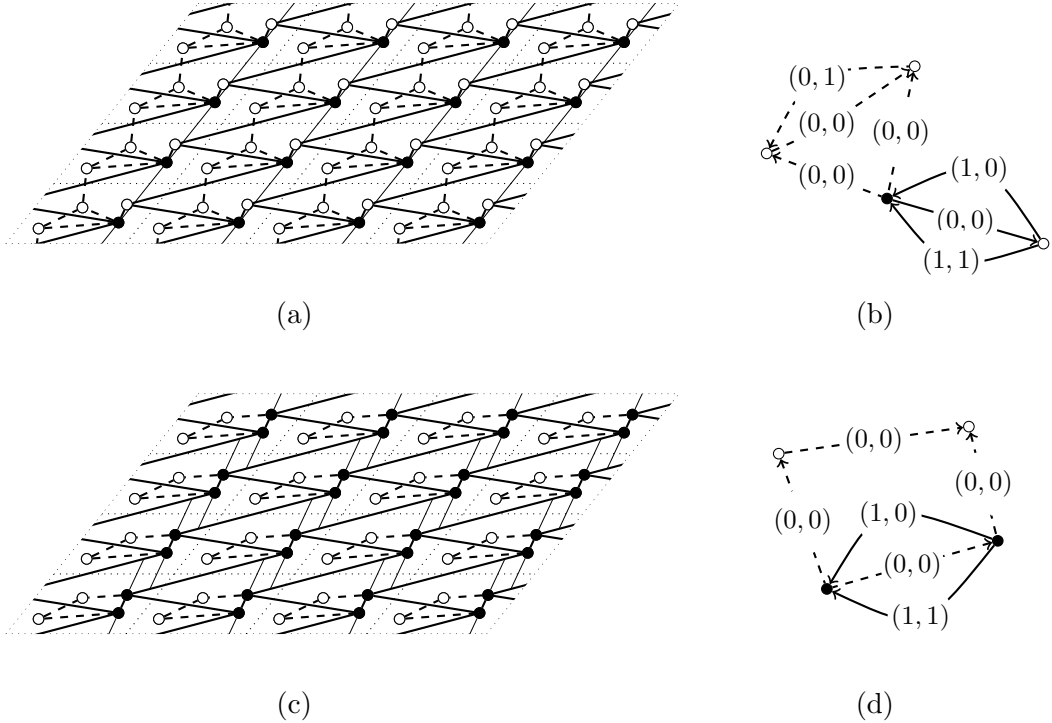

Consider, for example, the framework shown in Figure~\ref{fig:sep}(a) and its quotient $\mathbb{Z}^2$-labeled graph  \((G,\psi)\) shown in (b). The subgraph $H$ of \((G,\psi)\) induced by the dashed edges is a $(1,1)$-block with $V(H)\neq V(G)$. 
 Thus, by Lemma~\ref{lem:connectivity}, the framework in (a) is not globally \(L\)-periodically rigid.  Here \(\aff(P)\) is one of the thin black lines in (a) connecting the copies of the black vertices, and \(g\) is the reflection in \(\aff(P)\). The framework \((\tilde G,\tilde q)\) is obtained from \((\tilde G,\tilde p)\) by reflecting each connected component of the framework with dashed edges in the corresponding parallel copy of \(\aff(P)\) containing the black vertices of the component.

Figure~\ref{fig:sep}(c) shows another example of a framework  which is not globally \(L\)-periodically rigid by Lemma~\ref{lem:connectivity}.
Consider the corresponding quotient $\mathbb{Z}^2$-labeled graph  \((G,\psi)\) shown in Figure~\ref{fig:sep}(d). The subgraph $H$ induced by the dashed edges is a $(0,2)$-block of $G$ with $V(H)=V(G)$ but $0=s<k=2$. Here \(\aff(P)\) is one of the lines in (c) indicated by thin black line segments connecting pairs of black vertices, and  \(g\) is again the reflection in \(\aff(P)\). The framework \((\tilde G,\tilde q)\) is obtained from \((\tilde G,\tilde p)\)  as described in the previous case.

Sometimes Lemma~\ref{lem:connectivity} can be  strengthened by decomposing graphs.
Consider for example the case $d=2$. Suppose that $(G_1,\psi_1,p_1)$ is redundantly $L$-periodically rigid but contains a $(1,1)$-block $H$.
Then by Lemma~\ref{lem:connectivity} $(G_1,\psi_1,p_1)$ is not $L$-periodically globally rigid.
Now consider attaching a new $L$-periodically globally rigid framework $(G_2,\psi_2,p_2)$
at a vertex in $I(H)$ with $|V(G_1)\cap V(G_2)|=1$.
Then the resulting framework is clearly not $L$-periodically globally rigid but
$H$ is no longer a $(1,1)$-block and the resulting graph may satisfy the cut condition (and may also be redundantly $L$-periodically rigid).
In general, if a framework has a cut vertex, then we should look at each 2-connected component individually based on the following fact.

\begin{lemma}
\label{lem:glueing}
Let $(G,\psi, p)$ be a $\Gamma$-labeled framework with rank $d$ periodicity $\Gamma$, and
suppose that it can be decomposed into two frameworks $(G_i, \psi_i, p_i)\ (i=1,2)$ with
$|V(G_1)\cap V(G_2)|=1$.
Then $(G, \psi, p)$ is $L$-periodically globally rigid if and only if
each $(G_i, \psi_i, p_i)$ is $L$-periodically globally rigid and of rank $d$.
\end{lemma}
\begin{proof}
Note that, if the underlying periodicity group $\Gamma$ has rank $d$, then every $L(\Gamma)$-invariant isometry is a translation.
Hence the claim follows from Proposition~\ref{prop:isometry}.
\end{proof}
A similar statement to Lemma~\ref{lem:glueing} holds if we assume that the intersection of the two frameworks forms an $L$-periodically globally rigid subframework. Extending it to a more general gluing scenario (and sharpening the necessary condition for global periodic rigidity) is left as an open problem.

\subsection{The necessity of redundant $L$-periodic rigidity}\label{sec:z2nec}

Let  \(f:\mathbb{R}^d\rightarrow\mathbb{R}^k\) be a smooth map. Then \(x\in \mathbb{R}^d\) is said to be a \emph{regular point} of \(f\) if the Jacobian \(df|_x\) has maximum rank, and is a \emph{critical point} of \(f\) otherwise. Also \(f(x)\) is said to be a \emph{regular value} of \(f\) if, for all \(y\in f^{-1}(f(x))\), \(y\) is a regular point of \(f\). Otherwise \(f(x)\) is called a \emph{critical value} of \(f\).


We use the following lemmas.

\begin{lemma}\label{lem:edgefunction}(See, e.g., \cite{jjsz})
Let $f:\mathbb{R}^d\rightarrow \mathbb{R}^k$ be a polynomial map with rational coefficients
and $p$ be a generic point in $\mathbb{R}^d$.
If $df|_p$ is row-independent, then $f(p)$ is generic in $\mathbb{R}^k$.
\end{lemma}

\begin{lemma}(See, e.g., \cite{gortler2010})\label{lem:regularvalue}
Let $f:\mathbb{R}^d\rightarrow \mathbb{R}^k$ be a polynomial map with rational coefficients
and $p$ be a generic point in $\mathbb{R}^d$.
Then $f(p)$ is a regular value of $f$.
\end{lemma}

For a vector $p$ in $\mathbb{R}^d$, let $\mathbb{Q}(p)$ be the field generated by the entries of $p$ and the rationals.
For a field $F$ and an extension $K$, let $td[K:F]$ denote the transcendence degree of the extension.
For a field $K$, let $\overline{K}$ be the algebraic closure of $K$.
We also need the following lemma which will be used in the proof of  Lemma~\ref{lem:redrig}.
(See, e.g., \cite[Proposition 13]{jjt} for the proof.)
\begin{lemma}
\label{lem:pq}
Let $f:\mathbb{R}^d\rightarrow \mathbb{R}^d$ be a polynomial map with rational coefficients
and $p$ be a generic point in $\mathbb{R}^d$.
Suppose that $df|_p$ is nonsingular.
Then for every $q\in f^{-1}(f(p))$  we have $\overline{\mathbb{Q}(p)}=\overline{\mathbb{Q}(q)}$.
\end{lemma}

We now return to our discussion of $L$-periodically globally rigid frameworks.
Let $\Gamma$ be a group isomorphic to $\mathbb{Z}^k$, 
and $(G,\psi)$ be  a $\Gamma$-labeled graph with $|V|\geq t$, where $t=\max\{d-k, 1\}$. 
Let  $L:\Gamma\rightarrow \mathbb{R}^d$ be a nonsingular homomorphism, 
and for simplicity we suppose that
the linear span of $L(\Gamma)$ is $\{0\}^{d-k}\times \mathbb{R}^k$, the linear subspace spanned by the last $k$ coordinates.
We pick any $t$ vertices $v_1,\dots, v_t$, and define the augmented function of $f_{G,L}$ by $\hat{f}_{G,L}:=(f_{G,L}, g)$, where $g:\mathbb{R}^{d|V|}\rightarrow \mathbb{R}^{d+{t\choose2}}$ is a rational polynomial map
given by
\[
g(p)=(p_1(v_1), \dots, p_d(v_1),
p_1(v_2), \dots, p_{t-1}(v_2), p_1(v_3),\dots, p_{t-2}(v_3),\dots,  p_1(v_t))\qquad (p\in \mathbb{R}^{d|V|})
\]
 with $p_i(v_j)$ denoting the $i$-th coordinate of $p(v_j)$.
Augmenting $f_{G,L}$ by appending $g$ corresponds to ``pinning down'' some coordinates to eliminate trivial continuous motions.
\begin{proposition}
\label{prop:rigidity_matrix2}
Let $(G,\psi,p)$ be a  $\Gamma$-labeled framework in $\mathbb{R}^d$ with rank $k$ periodicity
and $L:\Gamma\rightarrow \mathbb{R}^d$ be a nonsingular homomorphism such that $L(\Gamma)=\{0\}^{d-k}\times \mathbb{R}^k$.
Suppose that $p$ is generic and $|V(G)|\geq t= \max\{d-k, 1\}$.
Then
\[
\rank d\hat{f}_{G,L}|_p=\rank df_{G,L}|_p+d+{t\choose 2}.
\]
\end{proposition}
\begin{proof}
We consider the system $(d\hat{f}_{G,L}|_p) \dot{p}=0$  of linear equations with variables $\dot{p}\in \mathbb{R}^{d|V|}$.
By regarding $\dot{p}\in \mathbb{R}^{d|V|}$ as a map $\dot{p}: V\rightarrow \mathbb{R}^d$, this system can be described as 
\begin{equation}
\label{eq:inf_motions}
\langle p(v_i)-(p(v_j)+L(\psi(v_iv_j)), \dot{p}(v_i)-\dot{p}(v_j) \rangle=0 \quad (v_iv_j\in E).
\end{equation}

By $L(\Gamma)=\{0\}^{d-k}\times \mathbb{R}^k$, an (orientation preserving) $L(\Gamma)$-invariant isometry is a composition of a rotation fixing the last $k$ coordinates and a translation. Hence a map $\dot{p}:V\rightarrow \mathbb{R}^d$ is a solution of the linear system (\ref{eq:inf_motions}) if it is of the form 
\[\dot{p}(v)=S(p(v)-p(v_1))+x\quad (v\in V)\] for some $x\in \mathbb{R}^d$ and some skew-symmetric matrix $S\in \mathbb{R}^{d\times d}$ such that 
only the top-left $(d-k)\times (d-k)$ block of $S$ may be nonzero and the remaining entries are zero. 
Such $\dot{p}$ is called a {\em trivial infinitesimal motion}.
The set of trivial infinitesimal motions  forms a linear space of dimension $d+{t\choose 2}$, and hence it suffices to prove that 
no trivial infinitesimal motion is  in the kernel of $d\hat{f}_{G,L}|_p$.

Let us take any trivial infinitesimal motion $\dot{p}$ described by $S$ and $x$ as above,
and suppose that $\dot{p}$ is in the kernel of $d\hat{f}_{G,L}|_p$.
Note that $d\hat{f}_{G,L}|_p$ is obtained from $df_{G,L}|_p$ by augmenting $dg|_p$.

By the first $d$ rows in  $dg|_p$, we have $x=0$. 
Similarly, if we denote  the $i$-th row of $S$ by $s_i$, then by the the remaining ${t\choose 2}$ rows of $dg|_p$, we get 
\begin{align*}
\langle s_1, p(v_j)-p(v_1)\rangle&=0 \quad (2\leq j\leq t) \\
\langle s_2, p(v_j)-p(v_1)\rangle&=0 \quad (2\leq j\leq t-1) \\
\vdots & \\ 
\langle s_{t-1}, p(v_j)-p(v_1)\rangle&=0 \quad (j=2)
\end{align*}
Since $p$ is generic and $S$ is a skew-symmetric matrix such that only the top-left $(d-k)\times (d-k)$ block of $S$ may be nonzero, 
we have $s_i=0$ for every $i$, implying $S=0$ and $\dot{p}=0$.
\end{proof}

We say that $(G,\psi,p)$ is {\em redundantly $L$-periodically rigid} if
$(G-e,\psi,p)$ is  $L$-periodically rigid for every $e\in E(G)$. (For simplicity, we slightly abuse notation here and denote the restriction of $\psi$ to $E(G)-e$ also by $\psi$.)

\begin{lemma}\label{thm:necweak}
Let $(G=(V,E),\psi,p)$ be a generic $\Gamma$-labeled framework in $\mathbb{R}^d$ with
rank $k$ periodicity $\Gamma$ and nonsingular $L:\Gamma\rightarrow \mathbb{R}^d$.
Suppose also that $|V|\geq d+1$ if $k\geq 1$ and $|V|\geq d+2$ if $k=0$.
If $(G, \psi,p)$ is $L$-periodically globally rigid,
then $(G,\psi,p)$ is redundantly $L$-periodically rigid.
\end{lemma}
\begin{proof}
The proof strategy is analogous to the one for Theorem 8.2 in \cite{jmn}.
Suppose for a contradiction that
$(G-e,\psi,p)$ is not  $L$-periodically  rigid for some $e\in E$.
Since $(G, \psi,p)$ is $L$-periodically globally rigid, it is $L$-periodically rigid.
Hence by Proposition~\ref{prop:rigidity_matrix} and Proposition~\ref{prop:rigidity_matrix2}
we have
\(\rank d\hat{f}_{G-e,L} |_p=d|V|-1\).
Since $p$ is generic, Lemma~\ref{lem:edgefunction} implies that
\begin{equation}
\label{eq:necweak}
td[\mathbb{Q}(\hat{f}_{G-e, L}(p)): \mathbb{Q}]\geq d|V|-1.
\end{equation}

\(\hat{f}_{G-e,L}(p)\) is a regular value of \(\hat{f}_{G-e, L}\) by Lemma \ref{lem:regularvalue}.
Hence its preimage is a 1-dimensional smooth manifold (see, e.g., \cite{milnor}).
Since this manifold is bounded (due to the ``pinning'' of some of the vertices) and closed, it is compact,
and it consists of a disjoint union of cycles by the classification of 1-dimensional manifolds.
Let  \(\mathcal{O}\) be the component that contains $p$.

Consider \(f_{e,L}:\mathbb{R}^{d|V|}\rightarrow\mathbb{R}\) which returns $\|p(u)-(p(v)+L(\psi(e))\|^2$ for the edge $e=uv$.
Since \(\hat{f}_{G-e, L}(p)\) is a regular value,  we have
 \(\rank df_{e,L}|_{p}=\rank d\hat{f}_{G,L}|_{p}-\rank d\hat{f}_{G-e,L}|_{p}=1\) 
 (see, e.g., \cite[Lemma 3.4]{jacksonkeevash} for the proof of the first equation).
 Hence $df_{e,L}|_{p}$ is nonzero (i.e., $p$ is not a critical point of \(f_{e,L}\)), and so
 the intermediate value theorem implies that there is a \(q\in\mathcal{O}\) with \(f_{e,L}(q)=f_{e,L}(p)\)
 and \(q\neq p\).
 We can assign an orientation to \(\mathcal{O}\) and we may assume that \(q\) is chosen as close to \(p\) as possible in the forward direction.

\(f_{e,L}(p)=f_{e,L}(q)\) implies that $\hat{f}_{G,L}(p)=\hat{f}_{G,L}(q)$.
This implies  $\hat{f}_{V,L}(p)=\hat{f}_{V,L}(q)$ since $(G,p)$ is $L$-periodically globally rigid.
By Proposition~\ref{prop:isometry}, $q$ can be written as $q=h\circ p$ for some $L(\Gamma)$-invariant isometry $h$.
Since $p(v_1)=q(v_1)$, there is an orthogonal matrix $S$ 
such that $q(u)=Sp(u)+(I-S)p(v_1)$ for every $u$ and  $S$ fixes each element in $L(\Gamma)$.

Take a path \(\gamma:[0,1]\ni t \mapsto p_t\in \mathcal{O}\) with \(\gamma(0)=p\) and \(\gamma(1)=q\),
and define a path \(\gamma':[0,1]\ni t \mapsto h\circ p_t\in \mathbb{R}^{d|V|}\).
Since $h$ fixes each element in $L(\Gamma)$, we have $f_{G,L}(p)=f_{G,L}(p_t)=f_{G,L}(h\circ p_t)$
for every $t\in[0,1]$.
In other words, $\gamma'$ is a path in ${\cal O}$.

If \(\gamma\) and \(\gamma'\) cover \(\mathcal{O}\) then we can assume that \(f_{e,L}\) increases as we pass through \(p\) in the forward direction. Then \(f_{e,L}\) has to increase as we pass through \(q\). Thus there are values \(t_1,t_2\) with \(0<t_1<t_2<1\) and \(f_{e,L}(p_{t_2})<f_{e,L}(p_1)=f_{e,L}(q)=f_{e,L}(p)=f_{e,L}(p_0)<f_{e,L}(p_{t_1})\). Using the intermediate value theorem, we then get a contradiction, because there exists a point \(p'\) with \(f_{e,L}(p')=f_{e,L}(p)\) between \(p_0\) and \(p_1\).

If \(\gamma\) and \(\gamma'\) do not cover \(\mathcal{O}\), then
there exists a $t\in [0,1]$ such that $\gamma(t)=\gamma'(t)$.
At this $t$, we have $p_t(u)=h(p_t(u))$ for every $u\in V$.
In other words, $p_t(V)$ is contained in the invariant subspace $H$ of $h$,
which is a proper affine subspace of $\mathbb{R}^d$ as $p\neq q$.
Let $d'(< d)$ be the affine dimension of $H$.
Since $H$ contains $L(\Gamma)$ whose basis is rational,
$H$ is determined by $(d'+1)d$  parameters, at most $(d'+1-k)d$ of which are independent over $\mathbb{Q}$.
Thus we get $td[\mathbb{Q}(p_t):\mathbb{Q}]\leq (d'+1-k)d+(|V|-(d'+1))d'$.
On the other hand, since $\mathbb{Q}(\hat{f}_{G-e,L}(p))=\mathbb{Q}(\hat{f}_{G-e,L}(p_t))\subseteq \mathbb{Q}(p_t)$, we also have  $td[\mathbb{Q}(p_t):\mathbb{Q}]\geq td[\mathbb{Q}(\hat{f}_{G-e, L}(p)): \mathbb{Q}]\geq d|V|-1$ by (\ref{eq:necweak}).
Thus $|V|\leq 1+d'+\frac{1-kd}{d-d'}\leq d+\frac{1-kd}{d-d'}$.
The last term is at most $d$ if $k\geq 1$ and at most $d+1$ if $k=0$,
which contradicts the assumption of the statement.
\end{proof}

\section{Characterizing Periodic Global Rigidity}
\label{sec:suf}
\subsection{Main theorems}
In this section we characterize periodic global rigidity in the plane based on the necessary conditions given in Section \ref{sec:nec}.
We need to  introduce one more term to describe the main theorem combinatorially.
Given a $\Gamma$-labeled graph $(G,\psi)$, Proposition~\ref{prop:rigidity_matrix} implies that
$(G, \psi, p)$ is $L$-periodically rigid for some generic $p$
if and only if $(G, \psi, p)$ is $L$-periodically rigid for every generic $p$.
Moreover, Theorem~\ref{thm:ross} says that the choice of $L$ is not important as long as $L$ is nonsingular.
In view of these facts, we say that $(G,\psi)$ is {\em periodically rigid} in $\mathbb{R}^d$ if
$(G,\psi,p)$ is $L$-periodically rigid for some (any) generic $p:V(G)\rightarrow \mathbb{R}^d$ and for some (any) nonsingular $L:\Gamma\rightarrow \mathbb{R}^d$.

We are now ready to state  our main theorems. 
\begin{theorem}
\label{thm:main0}
Let $(G,\psi, p)$ be a generic $\Gamma$-labeled framework in $\mathbb{R}^2$ with rank $k=1$ periodicity $\Gamma$ and $|V(G)|\geq 3$, and let $L:\Gamma\rightarrow \mathbb{R}^2$ be nonsingular.
Then $(G,\psi, p)$ is $L$-periodically globally rigid if and only if 
$(G,\psi)$ is redundantly periodically rigid in $\mathbb{R}^2$, 2-connected, and has no $(0,2)$-block.
 \end{theorem}
\begin{theorem}
\label{thm:main1}
Let $(G,\psi, p)$ be a generic $\Gamma$-labeled framework in $\mathbb{R}^2$ with rank $k=2$ periodicity $\Gamma$ and $|V(G)|\geq 3$, and let $L:\Gamma\rightarrow \mathbb{R}^2$ be nonsingular.
Then $(G,\psi, p)$ is $L$-periodically globally rigid if and only if $(G,\psi)$ is connected and each 2-connected component $(G', \psi')$ of $(G,\psi)$ is redundantly periodically rigid in $\mathbb{R}^2$,
has no $(0,2)$-block, and has rank two.
 \end{theorem}

Combining   Theorem~\ref{thm:main0} and Theorem~\ref{thm:main1} with Proposition~\ref{prop:trivial}, we have the main theorem in this paper.
\begin{theorem}\label{thm:main2}
A generic $L$-periodic framework $(\tilde{G}, \tilde{p})$ with at least three vertex orbits is $L$-periodically globally  rigid in the plane
if and only if its quotient $\Gamma$-labeled graph $(G,\psi)$ satisfies the combinatorial condition in Theorem~\ref{thm:jj}, Theorem~\ref{thm:main0} or Theorem~\ref{thm:main1} depending on the rank of the periodicity.
\end{theorem}

For smaller frameworks we have the following.
\begin{lemma}\label{rem:2orbits} 
Let $(G,\psi, p)$ be a generic $\Gamma$-labeled framework in $\mathbb{R}^2$ with rank $k$ periodicity $\Gamma$, and let $L:\Gamma\rightarrow \mathbb{R}^2$ be nonsingular. If $|V(G)|=1$, then $(G,\psi, p)$ is $L$-periodically globally rigid. If $|V(G)|=2$ and $k\geq 1$, then $(G,\psi, p)$ is $L$-periodically globally rigid if and only if the rank of $\Gamma_G$ is $k$. 
\end{lemma}
\begin{proof}  Let \(V(G)=\{u,v\}\), and suppose that every edge is oriented to \(v\). 

If \(k=1\), then there exist two edges from \(u\) to \(v\) with distinct labels. By switching, we may assume that  these labels are \(\id\) and \(\gamma\). Given \(q(v)\),  there are only two possible positions for \(q(u)\) if $(G,\psi, p)$ is equivalent to $(G,\psi,p)$ . In both cases, the resulting framework is congruent to $(G,\psi, p)$, and hence $(G,\psi, p)$ is $L$-periodically globally rigid.

 If \(k=2\), then there exists a third edge with label \(\gamma'\) not spanned by $\gamma$. 
 The given distance between \(q(u)\) and \(\gamma'q(v)\) then uniquely determines the position of \(q(u)\). Thus, $(G,\psi,q)$ is again congruent to $(G,\psi, p)$, and hence $(G,\psi, p)$ is $L$-periodically globally rigid.
\end{proof}

We remark here that there is an efficient algorithm to check whether \((G,\psi)\) satisfies the combinatorial conditions of the main theorem. 
After finding the 2-connected components of \(G\) the method given in \cite{jkt} can be used here as well to check their redundant periodic rigidity. 
Their rank can also be checked easily by finding an equivalent gain function by switchings (see, e.g., \cite{jkt}). 
Finally, for each pair of vertices we can  check whether they are the boundary of a (0,2)-block.

The proofs of Theorem~\ref{thm:main0} and Theorem~\ref{thm:main1} are almost identical and consist of  two parts, an algebraic part and a combinatorial part.
The algebraic part is solved in Lemma~\ref{lem:redrig}, and the combinatorial part is solved in Lemmas~\ref{lem:comb} and \ref{lem:comb2}.

\subsection{Algebraic part}
Let $(G,\psi,p)$ be a $\Gamma$-labeled framework.
For each edge $e$ from $v_i$ to $v_j$ in $G$, the edge direction $d_e$ is defined 
by $d_e=p(v_j)+L(\psi(e))-p(v_i)$.
We say that a vertex $v$ is {\em nondegenerate in $(G,\psi,p)$} if 
the set $\{d_e: \text{$e$ is incident to $v$}\}$ is linearly independent.
Note that this is a generic property, that is, $v$ is nondegenerate in a generic realization of $(G,\psi)$ if and only if it is nondegenerate in any generic realization of $(G,\psi)$.
Thus a vertex $v$ in $(G,\psi)$ is said to be {\em $d$-nondegenerate} if 
$v$ is nondegenerate in a generic realization of $(G,\psi)$.
%

Suppose that every edge incident to $v$ is directed from $v$.
For each pair of nonparallel edges $e_1=vu$ and $e_2=vw$ in $(G,\psi)$, 
let $e_1\cdot e_2$ be the edge from $u$ to $w$ with  label \(\psi(vu)^{-1}\psi(vw)\).
We define \((G_v,\psi_v)\) to be the $\Gamma$-labeled graph obtained from $(G,\psi)$ by removing $v$
and inserting $e_1\cdot e_2$ for every pair of nonparallel edges $e_1, e_2$ incident to $v$  
(unless an edge identical to $e_1\cdot e_2$  is already present in $(G,\psi)$).
The following is the periodic generalization of an observation given in \cite{jjsz,T1}.

\begin{lemma} \label{lem:redrig}
Let $(G,\psi,p)$ be a generic $\Gamma$-labeled framework with $|V(G)|\geq d+1$ and $L:\Gamma\rightarrow \mathbb{R}^d$ be nonsingular. Suppose that $G$ has a $d$-nondegenerate vertex $v$  of degree $d+1$ for which
\begin{itemize}
\item $(G-v,\psi)$ is $L$-periodically rigid in $\mathbb{R}^d$, and
\item $(G_v, \psi_v, p')$ is $L$-periodically globally rigid in $\mathbb{R}^d$ with notation \(p'=p|_{V(G)-v}\).
\end{itemize}
Then $(G, \psi, p)$ is $L$-periodically globally rigid in $\mathbb{R}^d$.
\end{lemma}
\begin{proof}
Pin the framework $(G,\psi,p)$ (as done in Section~\ref{sec:z2nec}) and take any $q\in \hat{f}_{G,L}^{-1}(\hat{f}_{G,L}(p))$.
Since $|V(G)|\geq d+1$, we may assume that $v$ is not ``pinned'' (i.e., $v$ is different from the vertices selected when augmenting $f_{G,L}$ to $\hat{f}_{G,L}$).
Our goal is to show that $p=q$.

Let $p'$ and $q'$ be the restrictions of $p$ and $q$ to $V(G)-v$, respectively.
Since $(G-v,\psi,p')$ is $L$-periodically rigid, we have
$\rank df_{G-v,L}|_{p'}=d|V(G-v)|-d-{d-k\choose 2}$ by Proposition~\ref{prop:rigidity_matrix}.
 We may assume (by rotating the whole space if necessary) that $L(\Gamma)=\{0\}^{d-k}\times \mathbb{R}^k$.
Then by Proposition~\ref{prop:rigidity_matrix2}, we further have 
$\rank d\hat{f}_{G-v,L}|_{p'}=d|V(G-v)|$. 
Thus we can take a spanning subgraph $H$ of $G-v$ such that $d\hat{f}_{H,L}|_{p'}$ has linearly independent rows, i.e. is nonsingular.
As $q'\in \hat{f}_{H,L}^{-1}(\hat{f}_{H,L}(p'))$,  it follows from Lemma~\ref{lem:pq} that $\overline{\mathbb{Q}(p')}=\overline{\mathbb{Q}(q')}$.
This in turn implies that $q'$ is generic.

We may assume that all the edges incident to $v$ are directed from $v$.
Let $e_0=vv_0, e_1=vv_1,\dots, e_{d}=vv_d$ denote the edges incident to $v$,
where $v_i=v_j$ may hold.
By switching we may assume $\psi(vv_0)=\textrm{id}$.
For each $1\leq i\leq d$, let
\begin{align*}
\gamma_i&=\psi(vv_i), \\
x_i&=p(v_i)+L(\gamma_i)-p(v_0), \\
y_i&=q(v_i)+L(\gamma_i)-q(v_0),
\end{align*}
and let $P$ and $Q$ be the $d\times d$-matrices whose $i$-th column is $x_i$ and $y_i$, respectively.
Note that since $v$ is $d$-nondegenerate  and $p'$, \(q'\) are generic, $x_1,\dots, x_d$ and $y_1,\dots, y_d$ are, respectively, linearly independent, and hence $P$ and $Q$ are both nonsingular.

Let $x_v=p(v)-p(v_0)$ and $y_v=q(v)-q(v_0)$.
We then have $\|x_v\|=\|y_v\|$ since $G$ has the edge $vv_0$ with $\psi(vv_0)=\textrm{id}$.
Due to the existence of the edge $e_i$ we also have
\begin{align*}0 &=
\langle p(v_i)+L(\gamma_i)-p(v), p(v_i)+L(\gamma_i)-p(v)\rangle  -
\langle q(v_i)+L(\gamma_i)-q(v),q(v_i)+L(\gamma_i)-q(v) \rangle\\
&=\langle x_i-x_v, x_i-x_v\rangle - \langle y_i-y_v, y_i-y_v\rangle\\
&=(\|x_i\|^2-\|y_i\|^2)-2\langle x_i, x_v\rangle +2\langle y_i, y_v\rangle,
 \end{align*}
 where we used $\|x_v\|=\|y_v\|$.
Denoting by $\delta$ the $d$-dimensional vector whose $i$-th coordinate is equal to $\|x_i\|^2-\|y_i\|^2$,
the above $d$ equations can be summarized as
$$0=\delta-2P^T x_v+2Q^T y_v$$ which is equivalent to
$$y_v=(Q^T)^{-1}P^T  x_v -\frac{1}{2}(Q^T)^{-1}\delta.$$
By putting this into  $\|x_v\|^2=\|y_v\|^2$, we obtain
\begin{equation}
\label{eq:algebraic}
x_v^T(I_d-PQ^{-1}(PQ^{-1})^T)x_v-(\delta^T Q^{-1}(Q^{-1})^T P^T)x_v+\frac{1}{4}\delta^T Q^{-1}(Q^{-1})^T\delta=0,
\end{equation}
where $I_d$ denotes the $d\times d$ identity matrix.

Note that each entry of $P$ is contained in $\mathbb{Q}(p')$, and each entry of $Q$ is contained in $\mathbb{Q}(q')$. Since $\overline{\mathbb{Q}(p')}=\overline{\mathbb{Q}(q')}$,
this implies that each entry of $PQ^{-1}$ is contained in $\overline{\mathbb{Q}(p')}$.
On the other hand, since $p$ is generic, the set of coordinates  of $p(v)$ (and hence those of $x_v$) is algebraically independent over $\overline{\mathbb{Q}(p')}$.
Therefore, by regarding the left-hand side of (\ref{eq:algebraic}) as a polynomial in $x_v$,
the polynomial must be identically zero. In particular, we get
$$I_d-PQ^{-1}(PQ^{-1})^T=0.$$
Thus, $PQ^{-1}$ is orthogonal.
In other words, there is some orthogonal matrix $S$ such that $P=SQ$,
and we get $\|p(v_i)+L(\gamma_i)-p(v_0)\|=\|x_i\|=\|Sy_i\|=\|y_i\|=\|q(v_i)+L(\gamma_i)-q(v_0)\|$ for every $1\leq i\leq d$.
Therefore, $q'\in f_{G_v, L}^{-1}( f_{G_v, L}(p'))$.
Since  $(G_v, \psi_v,p)$ is $L$-periodically globally rigid, this in turn implies that
$f_{V-v, L}(p')=f_{V-v, L}(q')$.
Thus we have $p'=q'$.

Since $v$ is $d$-nondegenerate  and $p$ is generic,
$\bigcup_{v_i\in N_G(v)}\{p(v_i)+L(\psi(e)):\, e=vv_i\in E(G)\}$ affinely spans $\mathbb{R}^d$.
Hence there is a unique extension of $p':V-v\rightarrow \mathbb{R}^d$ to $r:V\rightarrow \mathbb{R}^d$ such that $f_{G, L}(r)=f_{G,L}(p)$.
Thus we obtain $p=q$.
\end{proof}

\subsection{Combinatorial part}
The combinatorial part consists of the following two lemmas whose proof will be given  in the next sections separately.
\begin{lemma}\label{lem:comb}
Let $(G,\psi)$ be a $\Gamma$-labeled graph with rank $k\geq 1$ periodicity $\Gamma$.
Suppose that $(G,\psi)$ is 2-connected, redundantly periodically rigid in $\mathbb{R}^2$, and has no $(0,2)$-block. Then at least one of the following holds:
\begin{itemize}
\item[(i)] There exists  $e\in E(G)$ such that the \(\Gamma\)-labeled graph $(G-e,\psi)$ is 2-connected, redundantly periodically rigid, and has no $(0,2)$-block.
\item[(ii)] $G$ has  a vertex of degree three.
\end{itemize}
\end{lemma}

\begin{lemma}\label{lem:comb2}
Let $(G,\psi)$ be a $\Gamma$-labeled graph with rank $k\geq 1$ periodicity $\Gamma$ and \(|V(G)|\geq4\).
Suppose that $(G,\psi)$ is 2-connected, redundantly periodically rigid in $\mathbb{R}^2$, and has no $(0,2)$-block.
Then the minimum degree of $G$ is at least three and the following hold for every vertex $v$ of degree three.
\begin{itemize}
\item $v$ is $2$-nondegenerate.
\item $(G-v, \psi)$ is periodically rigid in $\mathbb{R}^2$.
\item $(G_v, \psi_v)$ is 2-connected, redundantly periodically rigid in $\mathbb{R}^2$, and has no $(0,2)$-block.
\end{itemize}
\end{lemma}

\subsection{Proofs of Theorem~\ref{thm:main0} and Theorem~\ref{thm:main1}}
Assuming the correctness of Lemma~\ref{lem:comb} and Lemma~\ref{lem:comb2}
we are now ready to prove our main theorems.
\begin{proof}[Proof of Theorem~\ref{thm:main0}]
To see the necessity, first observe the following:
\begin{equation}
\label{eq:nec1}
\begin{split}
&\text{$(G,\psi)$  has no $(s,t)$-block $H$ with $s+t\leq 2=d$ such that $V(H)\neq V(G)$ } \\
&\text{or $s<1=k$ if and only if $(G,\psi)$ is  2-connected and has no $(0,2)$-block.}
\end{split}
\end{equation}
Indeed, if $(G,\psi)$ is not 2-connected, then it has an $(s,t)$-block $H$ with $s\leq k=1$, $t\leq 1$, and $V(H)\neq V(G)$.
 Conversely, suppose that  $(G,\psi)$ has an $(s,t)$-block $H$ with $s+t\leq 2$ such that $V(H)\neq V(G)$ or $s<k=1$.
 If $H$ is not a $(0,2)$-block, then $s\geq 1$, $t\leq 1$ and $V(H)\neq V(G)$ hold. Hence $G$ is not 2-connected.
  Thus by (\ref{eq:nec1}) the necessity of the conditions follows from Lemmas~\ref{lem:connectivity} and \ref{thm:necweak}.

The proof of the sufficiency is done by induction on the lexicographic order of the list $(V(G), E(G))$.
Suppose that $|V(G)|=3$.
By the 2-connectivity, \(G\) contains a triangle and the minimum degree in \(G\) is at least three.
Let $G'$ be an inclusionwise minimal spanning subgraph that is 2-connected and redundantly periodically rigid. Then it consists of five edges,  two parallel classes and one simple edge.
Since $G'$ has a vertex of degree three, we can use Lemmas~\ref{lem:redrig} and~\ref{rem:2orbits} to deduce that \((G',\psi)\) (and hence $(G,\psi)$) is $L$-periodically globally rigid.

Assume that $|V(G)|\geq 4$.
Suppose that $G$ has a vertex of degree three.
By Lemma~\ref{lem:comb2}, $(G_v, \psi_v)$ is 2-connected, redundantly periodically rigid, and has no $(0,2)$-block. 
Hence, by induction, $(G_v, \psi_v, p)$ is $L$-periodically globally rigid.
Therefore it follows from Lemma~\ref{lem:comb2} and Lemma~\ref{lem:redrig} that $(G,\psi, p)$ is $L$-periodically globally rigid.

Thus we may assume that $G$ has no vertex of degree three.
Then by Lemma~\ref{lem:comb} there is an edge $e$ such that $(G-e, \psi)$ is 2-connected, redundantly periodically rigid, and has no $(0,2)$-block.
By induction,  $(G-e,\psi,p)$ (and hence $(G,\psi,p)$) is $L$-periodically globally rigid.
\end{proof}

\begin{proof}[Proof of Theorem~\ref{thm:main1}]
By Lemma~\ref{lem:glueing} we may assume that $G$ is 2-connected.
Then the necessity of the conditions again follows from Lemmas~\ref{lem:connectivity} and \ref{thm:necweak}.

The proof of the sufficiency is similar to that of Theorem~\ref{thm:main0}, and it is done by induction on the lexicographic order of the list $(V(G), E(G))$.
The case when $|V(G)|=3$ is exactly the same as that for Theorem~\ref{thm:main0}.
Hence we assume $|V(G)|\geq 4$.

Suppose that $G$ has a vertex of degree three.
By Lemma~\ref{lem:comb2}, $(G_v, \psi_v)$ is 2-connected, redundantly periodically rigid, and has no $(0,2)$-block. Also, by the definition of $(G_v,\psi_v)$, we have $\Gamma_{G_v}=\Gamma_G$. 
Hence, by induction, $(G_v, \psi_v, p)$ is $L$-periodically globally rigid.
Therefore it follows from Lemma~\ref{lem:comb2} and Lemma~\ref{lem:redrig} that $(G,\psi, p)$ is $L$-periodically globally rigid.

Thus we may assume that $G$ has no vertex of degree three.
Then by Lemma~\ref{lem:comb} there is an edge $e$ such that $(G-e, \psi)$ is 2-connected, redundantly periodically rigid, and has no $(0,2)$-block.
If the rank of $\Gamma_{G-e}$ is equal to two, then we can apply the induction hypothesis, meaning that  $(G-e,\psi,p)$ (and hence $(G,\psi,p)$) is $L$-periodically globally rigid.
Therefore, we assume that the rank of $\Gamma_{G-e}$ is smaller than two.

Since $\Gamma_G$ has rank $2$, the rank of $\Gamma_{G-e}$ is nonzero.
 Thus the rank of $\Gamma_{G-e}$ is one.
By switching, we may suppose that every label in $E(G-e)$ is in $\Gamma_{G-e}$
and consider the restriction $L':\Gamma_{G-e}\rightarrow \mathbb{R}^2$ of $L$.
By Theorem~\ref{thm:main0}, $(G-e,\psi,p)$ is $L'$-periodically globally rigid.
To show that $(G,\psi,p)$ is $L$-periodically globally rigid,
take any $q:V(G)\rightarrow \mathbb{R}^2$ such that $f_{G,L}(p)=f_{G,L}(q)$.
Since $(G-e,\psi,p)$ is $L'$-periodically globally rigid, Proposition~\ref{prop:isometry} implies that $q$ can be written as $q=h\circ p$ for some $L(\Gamma_{G-e})$-invariant isometry $h$.
Since the rank of $\Gamma_{G-e}$ is one and the dimension of the ambient space is two,
$h$ is either the identity or the reflection along the line whose direction is in the span of $L(\Gamma_{G-e})$.

Let $i$ and $j$ be the end vertices of $e$ and let $\gamma_e=\psi(e)$.
Using the orthogonal matrix $S$ representing the reflection along the span of $L(\Gamma_{G-e})$,
we have  $q(i)-q(j)=S(p(i)-p(j))$.
By looking at the length constraint for $e$,
$\|p(i)-(p(j)+L(\gamma_e))\|=\|q(i)-(q(j)+L(\gamma_e))\|$, which implies
$\langle (I-S)(p_i-p_j), L(\gamma_e)\rangle=0$,
or equivalently,  $\langle p_i-p_j, (I-S^{\top})L(\gamma_e)\rangle=0$.
Note that $S$ is determined by $L(\Gamma_{G-e})$, and independent from $p_i-p_j$.
Hence, the generic assumption for $p$ implies $(I-S^{\top})L(\gamma_e)=0$.
Since $\gamma_e\notin \Gamma_{G-e}$, this implies $I=S$.
Thus we get $f_{V,L}(p)=f_{V,L}(q)$, and $(G,\psi, p)$ is $L$-periodically globally rigid.
\end{proof}

\section{Proof of the Combinatorial Part: Proofs of Lemma~\ref{lem:comb} and Lemma~\ref{lem:comb2}}
Due to Theorem~\ref{thm:ross},  we may  work in a purely combinatorial world in order to prove Lemma~\ref{lem:comb} and Lemma~\ref{lem:comb2}.

Since all the graphs we treat in the following discussions are $\Gamma$-labeled graphs,
we shall use the following convention.
Throughout the section, we  omit the labeling function $\psi$ to denote a $\Gamma$-labeled graph $(G,\psi)$.
The underlying graph of each $\Gamma$-labeled graph is directed, but the direction is used only to refer to the group labeling. So the underlying graph is treated as an undirected graph if we are interested in its graph-theoretical properties, such as the connectivity, vertex degree, and so on.

Let $G=(V,E)$ be a $\Gamma$-labeled graph.
For disjoint sets $X, Y\subseteq V$, $d_G(X,Y)$ denotes the number of edges between $X$ and $Y$,
and let $d_G(X):=d_G(X, V\setminus X)$ and $d_G(v)=d_G(\{v\})$ for $v\in V$.
For $X\subseteq V$, $i_G(X)$ denotes the number of edges induced by $X$.
For an edge set $F$, let $G[F]=(V(F), F)$.

Given two $\Gamma$-labeled graphs $G_1=(V_1, E_1)$ and $G_2=(V_2, E_2)$,
let $G_1\cup G_2=(V_1\cup V_2, E_1\cup E_2)$ and
$G_1\cap G_2=(V_1\cap V_2, E_1\cap E_2)$.
Note that $G_1\cup G_2$ and $G_1\cap G_2$ are $\Gamma$-labeled graphs whose labeling are inherited from those of $G_1$ and $G_2$.
Note also that, when taking the union or the intersection,
 the labels are taken into account and two edges are recognized as the same edge  if and only if they are identical. Hence $G_1\cup G_2$ may contain parallel edges even if $G_1$ and $G_2$ are simple.

Several terms defined for edge sets will be  used for graphs $G$ by implicitly referring to $E(G)$. If there is no confusion, terms for graphs will be conversely used for edge sets $E$ by referring to the graph $(V(E), E)$.

We use the following terminology from matroid theory.
Let  \(\M=(E,\mathcal{I})\) be a matroid, where $E$ is the ground set and $\mathcal{I}$ is the family of independent sets.  A circuit in $\M$ is a minimal dependent subset of $E$, i.e. a dependent set whose proper subsets are all independent. We define a relation on \(E\) by saying that \(e,f\in E\) are \emph{related} if \(e=f\) or if there is a circuit \(C\) in \(\M\) with \(e,f\in C\). It is well-known that this is an equivalence relation. The equivalence classes are called the \emph{components} of \(\M\). If \(\M\) has at least two elements and only one component then \(\M\) is said to be \emph{connected}.
If \(\M\) has components \(E_1,E_2,\dots,E_t\) and \(\M_i\) is the matroid restriction of \(\M\) onto \(E_i\) then \(\M=\M_1\oplus\M_2\oplus\dots\oplus\M_t\) is the direct sum of the \(\M_i\)'s.

\subsection{Count matroids}
Let $V$ be a finite set and $\Gamma$ be  a  group isomorphic to $\mathbb{Z}^k$.
We say that an edge set $E$ is {\em independent} if $|F|\leq 2|V(F)|-3$ for every nonempty balanced $F\subseteq E$ and  $|F|\leq 2|V(F)|-2$ for every nonempty $F\subseteq E$.
Proposition~\ref{prop:rigidity_matrix} and Theorem~\ref{thm:ross} imply that the family of all independent edge sets forms the family of independent sets of a matroid on $K(V,\Gamma)$,
which is denoted by ${\cal R}_2(V, \Gamma)$ or simply by ${\cal R}_2(V)$.
(This can also be checked in a purely combinatorial fashion, see, e.g., \cite{jkt}.)
Let $r_2$ be the rank function and ${\rm cl}_2$ be the closure operator of ${\cal R}_2(V)$.

We say that $G$ is a {\em circuit} if $E(G)$ is a circuit in ${\cal R}_2(V(G))$.
An example of a balanced circuit is the simple complete graph on four vertices with identity label on each edge.
An example of an unbalanced circuit is a $\Gamma$-labeled graph with two vertices and three parallel edges such that no two parallel edges form a balanced set.

We say that $G$ is {\em $M$-connected} if the restriction of ${\cal R}_2(V(G))$ to $E(G)$ is connected.

For a $\Gamma$-labeled graph $G$ on $V$, denote $r_2(G)=r_2(E(G))$. 
By Theorem~\ref{thm:ross}, $G$ is periodically rigid in $\mathbb{R}^2$ if and only if
$r_2(G)=2|V(G)|-2$.

Note that if $G$ is balanced, then the restriction of ${\cal R}_2(V)$ to $E(G)$ is isomorphic to the 2-dimensional generic rigidity matroid.
Hence we say that  $G$ is {\em rigid} if $G$ is balanced and $r_2(G)=2|V(G)|-3$.

We can use known properties of the 2-dimensional generic rigidity matroids for balanced sets.
The following statements  can be found in \cite{jj} (with a slightly different terminology).
\begin{lemma}
\label{lem:4}
If a $\Gamma$-labeled graph $G$ is balanced and $M$-connected, then $G$ is rigid.
\end{lemma}
\begin{lemma}\label{cl:G_1cupG_2rigid1}
Let $G_1$ and $G_2$ be two rigid graphs.
Suppose that   $|V(G_1)\cap V(G_2)|\geq 2$ and  $G_1\cup G_2$  is balanced.
Then $G_1\cup G_2$ is rigid.
\end{lemma}

The following simple property was first observed in \cite{ross2}.
\begin{lemma}[\cite{ross2}]
\label{lem:0extension}
Let $v$ be a vertex of degree two in $G$ with $|V(G)|\geq 2$.
Then $r_2(G-v)=r_2(G)-2$.
\end{lemma}

Lemma~\ref{lem:0extension} enables us to show an analogue of Lemma~\ref{cl:G_1cupG_2rigid1}.
\begin{lemma}\label{cl:G_1cupG_2rigid2}
Let $G_1$ and $G_2$ be $\Gamma$-labeled graphs.
\begin{itemize}
\item[(i)] If $G_1$ is periodically rigid, $G_2$ is rigid and $|V(G_1)\cap V(G_2)|\geq2$, then $G_1\cup G_2$ is periodically rigid.
\item[(ii)] If $G_1$ and $G_2$ are periodically rigid and $|V(G_1)\cap V(G_2)|\geq 2$,
then $G_1\cup G_2$ is periodically rigid.
\end{itemize}
\end{lemma}
\begin{proof}
Let $G_i=(V_i, E_i)$ for $i=1,2$.
To see the first claim, suppose that $G_1$ is periodically rigid and $G_2$ is rigid.
Since $G_2$ is rigid, $G_2$ is balanced, and we may suppose, by switching (see \cite[Proposition 2.3 and Lemma 2.4]{jkt}), that $\psi(e)=\textrm{id}$ for every edge $e$ of $G_2$.
Thus we have $K^0(V_2)\subseteq {\rm cl}_2(E_2)$,
where $K^0(V_2)$ denotes the set of edges on $V_2$ whose labels are the identity.
Since $|V_1\cap V_2|\geq2$, there are two distinct vertices $u$ and $v$ in $V_1\cap V_2$.
Let $F$ be the edge set of the complete bipartite subgraph of $K^0(V_2)$
whose one partite set is $\{u,v\}$ and whose other partite set is $V_2\setminus V_1$.
By Lemma~\ref{lem:0extension}, we have $r_2(E_1\cup F)=r_2(E_1)+2|V_2\setminus V_1|$.
Since $F\subseteq K^0(V_2)\subseteq {\rm cl}_2(E_2)$, we get
$$r_2(E_1\cup E_2)\geq r_2(E_1\cup F)=2|V_1|-2+2|V_2\setminus V_1|=2|V_1\cup V_2|-2.$$
Thus $G_1\cup G_2$ is periodically rigid.

We can use the same argument to prove the second statement. 
In this case we have $K^0(V_2)\subset K(V_2,\Gamma)\subseteq {\rm cl}_2(E_2)$ and we can again use Lemma~\ref{lem:0extension} to deduce that $r_2(E_1\cup E_2)=2|V_1\cup V_2|-2$, which completes the proof.
\end{proof}

Lemma~\ref{lem:4} implies that  a balanced circuit is always rigid. By Theorem~\ref{thm:ross}, the following periodic version can be easily proved.
\begin{lemma}\label{lem:circuit_periodic_rigid}
If $G$ is an unbalanced circuit, then it is periodically rigid.
\end{lemma}

\subsection{$M$-connectivity and ear decomposition}

Jackson and Jord{\'a}n~\cite{jj} used ear decompositions of connected rigidity matroids as a key tool in their proof of Theorem~\ref{thm:jj}.
Let \(C_1,C_2,\dots,C_t\) be a non-empty sequence of circuits of the matroid \(\M\).
 For \(1\leq j\leq t\), we denote \(D_j=C_1\cup C_2\cup\dots\cup C_j\) and define the {\em lobe} $\tilde{C}_j$ of $C_j$ by $\tilde{C}_j=C_j\setminus D_{j-1}$. We say that \(C_1,C_2,\dots,C_t\) is a \emph{partial ear decomposition} of \(\M\) if for every \(2\leq i\leq t\) the following properties hold:
\begin{itemize}
\item[(E1)] \(C_i\cap D_{i-1}\neq\emptyset\),
\item[(E2)] \(C_i\setminus D_{i-1}\neq\emptyset\),
\item[(E3)] no circuit \(C'\) satisfying (E1) and (E2) has \(C'\setminus D_{i-1}\) properly contained in \(C_i\setminus D_{i-1}\).
\end{itemize}
An \emph{ear decomposition} of \(\M\) is a partial ear decomposition with \(D_t=E\). We need the following facts about ear decompositions.
\begin{lemma}[\cite{coullard,jj}]\label{lem:EDfacts}
Let \(\M\) be a matroid. Then the following hold.
\begin{enumerate}
  \item[(a)] \(\M\) is connected if and only if \(\M\) has an ear decomposition.
	\item[(b)] If \(\M\) is connected then any partial ear decomposition of \(\M\) can be extended to an ear decomposition of \(\M\).
	\item[(c)] If \(C_1,C_2,\dots,C_t\) is an ear decomposition of \(\M\) then \(r(D_i)-r(D_{i-1})=|\tilde{C}_i|-1\) for \(2\leq i\leq t\).
\end{enumerate}
\end{lemma}

For an $M$-connected graph $G=(V,E)$,  {\em an ear decomposition of $E$} means an ear decomposition of the restriction of ${\cal R}_2(V)$ to $E$.

When the matroid ${\cal M}$ is specialized to the generic 2-dimensional rigidity matroid, several properties of ear-decompositions were proved in \cite{jj,j}.
We will use the following (with some of the terminology adjusted to fit our context).
\begin{lemma}[\cite{jj}]
\label{lem:jj}
Let $G$ be balanced $M$-connected,
and let $C_1, \dots, C_t$ be an ear decomposition of $E(G)$  with $t\geq 2$.
Let $Y=V(C_t)\setminus \bigcup_{i=1}^{t-1} V(C_i)$.
Then the following hold.
\begin{description}
\item[(a)] Either $Y=\emptyset$ and $|\tilde{C}_t|=1$, or $Y\neq \emptyset$ and every edge $e\in \tilde{C}_t$ is incident to $Y$.
\item[(b)] $|\tilde{C}_t|=2|Y|+1$.
\end{description}
\end{lemma}

In \cite[Lemma 3.4.3]{j}, Jord{\'a}n pointed out that Lemma~\ref{lem:jj} immediately implies the existence of a degree three vertex in a minimally $M$-connected graph. His proof actually gives a slightly stronger statement as follows.
\begin{lemma}
\label{lem:jordan}
Let $G$ be balanced and  $M$-connected,
and let $C_1,\dots, C_t$ be an ear decomposition of $E(G)$.
Let $Y=V(C_t)- V(\bigcup_{i=1}^{t-1} C_i)$.
Suppose further that $|\tilde{C}_t|>1$. 
Then $Y$ contains a vertex which has degree three in $G$.
Moreover, if $|Y|\geq 2$, then $Y$ contains at least two vertices of degree three in $G$.
\end{lemma}
\begin{proof}
It is a well known fact that a circuit in the ordinary 2-dimensional generic rigidity matroid has at least four vertices of degree three. 
Hence the claim follows if $t=1$. 
Suppose that $t\geq 2$.
Lemma~\ref{lem:jj}(a) says that, if  $|\tilde{C}_t|>1$, then  $Y\neq \emptyset$
and every edge in $\tilde{C}_t$ is incident to $Y$.
Moreover, Lemma~\ref{lem:jj}(b) says that
\begin{equation}
\label{eq:jordan1}
|\tilde{C}_t|=2|Y|+1.
\end{equation}
Hence if $|Y|=1$ then the vertex in $Y$ has degree three.

If $|Y|\geq 2$, then the edge set of $G$ induced by $Y$ is a proper subset of $\tilde{C}_t$, and hence $i_G(Y)\leq 2|Y|-3$. Combining this with (\ref{eq:jordan1}), we have that
the total degree of the vertices of $Y$ is
$d_G(Y, V(G)\setminus Y)+2i_G(Y)=|\tilde{C}_t|+i_G(Y)=
4|Y|-2$.
Thus $Y$ contains  at least two vertices that have  degree three in $G$.
\end{proof}

It is well known that every circuit in the generic rigidity matroid (i.e., every balanced circuit) is $2$-connected and $3$-edge-connected.
The next lemma shows that even if $G$ is unbalanced or  $M$-connected, we can guarantee the  same connectivity.
 \begin{lemma} \label{lem:7}
If $G$ is $M$-connected, then $G$ is 2-connected and $3$-edge-connected.
\end{lemma}
\begin{proof}
The statement is known if $G$ is balanced (see, e.g., \cite{jj}).
Hence we assume that $G$ is unbalanced.

It suffices to consider the case when $G$ is an unbalanced circuit.
Suppose that $G$ can be decomposed into two edge-disjoint subgraphs $G_1$ and $G_2$
with $|V(G_1)\cap V(G_2)|\leq 1$  and $|V(G_i)|\geq 2$ for $i=1,2$.
Then $2|V(G)|-1=|E(G)|=|E(G_1)|+|E(G_2)|\leq 2|V(G_1)|-2+2|V(G_2)|-2
\leq 2|V(G)|-2$, which is a contradiction.
Thus $G$ is 2-connected.

The same counting argument works for showing that $G$ is 3-edge-connected.
Suppose that there are two vertex-disjoint induced subgraphs $G_1$ and $G_2$
connected by at most two edges in $G$ with $V(G_1)\cup V(G_2)=V(G)$.
Then we have $2|V(G)|-1=|E(G)| \leq |E(G_1)|+|E(G_2)| +2\leq 2|V(G_1)|-2 + 2|V(G_2)|-2 +2=2|V(G)|-2$, which is a contradiction.
\end{proof}

\subsection{Redundant rigidity and M-connectivity}\label{subsec:red_mconn}

Based on the ear decomposition, in this section we shall reveal a connection between redundant rigidity and $M$-connectivity. In particular, we show that, under a certain connectivity condition, these properties are equivalent. This is an analogue of the fact that $M$-connectivity is equivalent to redundant rigidity under (near) 3-connectivity in the plane~\cite[Theorem 3.2]{jj}.

In Lemma~\ref{lem:3} we first give the unbalanced version of Lemma~\ref{lem:4}.
For the proof we need the following technical lemma.

\begin{lemma}\label{lem:disconnected}
Let $E_1$ and $E_2$ be balanced edge sets in a $\Gamma$-labeled graph $G$.
Suppose that $E_1\cup E_2$ is unbalanced.
Then the graph $(V(E_1)\cap V(E_2), E_1\cap E_2)$ is disconnected.
\end{lemma}
\begin{proof}
This is a special case of \cite[Lemma 2.4]{jkt}.
Since the proof is easy, we include it for completeness.
Suppose that $(V(E_1)\cap V(E_2), E_1\cap E_2)$ is connected.
Then one can take a spanning tree $T$ in $(V(E_1\cup E_2), E_1\cup E_2)$
such that $T\cap E_i$ is a spanning tree in \((V(E_i),E_i)\) for \(i=1,2\).
By switching, we may assume that $\psi(e)=\textrm{id}$ for every $e\in T$ (see \cite[Proposition 2.3]{jkt}).
Since $E_i$ is balanced, we get $\psi(e)=\textrm{id}$ for every $e\in E_1\cup E_2$,
contradicting that $E_1\cup E_2$ is unbalanced.
\end{proof}

\begin{lemma} \label{lem:3}
If $G$ is unbalanced and M-connected, then $G$ is periodically rigid.
\end{lemma}
\begin{proof}
Since $G$ is M-connected,  $E(G)$ has an ear decomposition $C_1,\ldots, C_t$.
It follows from $|D_i\cap C_{i+1}|\geq1$ that there are at least two vertices in \(V(D_i)\cap V(C_{i+1})\). We will prove by induction on $i$
that 
\begin{equation}
\label{eq:3-0}
\text{$D_i$ is rigid if $D_i$ is balanced, and $D_i$ is periodically rigid otherwise.}
\end{equation}
This follows from Lemma~\ref{lem:4} and Lemma~\ref{lem:circuit_periodic_rigid} if $i=1$. Hence, we assume $i\geq 2$.

If $D_{i-1}$ is unbalanced, then by induction $D_{i-1}$ is periodically rigid.
By Lemma \ref{cl:G_1cupG_2rigid2}, $D_{i-1}\cup C_i$ is periodically rigid, implying (\ref{eq:3-0}).

If $D_{i-1}$ is balanced, then by induction it is rigid. If $C_i$ is unbalanced, then by Lemma~\ref{cl:G_1cupG_2rigid2}(i) and Lemma~\ref{lem:circuit_periodic_rigid},  $D_{i-1}\cup C_i$ is periodically rigid.
If $C_i$ is balanced and $D_{i-1}\cup C_i$ is balanced, then by Lemma~\ref{lem:4} and Lemma~\ref{cl:G_1cupG_2rigid1}, $D_{i-1}\cup C_i$ is rigid.

For (\ref{eq:3-0}) it remains to show \(r_2(D_{i})=2|V(D_{i})|-2\) if $C_i$ and \(D_{i-1}\) are balanced and  $D_{i-1}\cup C_i$ is unbalanced.
Let $s:=2|V(C_i\cap D_{i-1})|-3-|C_i\cap D_{i-1}|$.
We claim that 
\begin{equation}
\label{eq:lem3}
s+2( |V(C_i)\cap V(D_{i-1})|-|V(C_i\cap D_{i-1})|)\geq 1.
\end{equation}
To see this, observe first that $s\geq 0$ holds since $C_i\cap D_{i-1}$ is balanced and independent.
Also (\ref{eq:lem3}) trivially holds if $s\geq 1$. 
Hence suppose $s=0$. Then $C_i\cap D_{i-1}$ is rigid as $C_i\cap D_{i-1}$ is balanced, and hence $C_i\cap D_{i-1}$ is connected.
However, since $(V(C_i)\cap V(D_{i-1}), C_i\cap D_{i-1})$ is disconnected by Lemma~\ref{lem:disconnected},
we have $V(C_i)\cap V(D_{i-1})\neq V(C_i\cap D_{i-1})$, implying (\ref{eq:lem3}) even for $s=0$.


Since $D_{i-1}$ is rigid and $C_i$ is a balanced circuit, Lemma~\ref{lem:EDfacts}(c) implies
\begin{align*}
r_2(D_i)&=r_2(D_{i-1})+|C_i\setminus D_{i-1}|-1 \\
&=2|V(D_{i-1})|-3+|C_i|-|C_i\cap D_{i-1}|-1 \\
&=2|V(D_{i-1})|-3+2|V(C_i)|-2-2|V(C_i\cap D_{i-1})|+3+s-1 \\
&=2(|V(D_{i-1})|+|V(C_i)|-|V(C_i\cap D_{i-1})|)-3+s \\
&=2|V(D_i)|-3+2(|V(C_{i})\cap V(D_{i-1})|-|V(C_i\cap D_{i-1})|)+s\\
&\geq 2|V(D_i)|-2
\end{align*}
where the last inequality follows from (\ref{eq:lem3}).

Since $E(G)=D_t$ and $G$ is unbalanced, we conclude that $G$ is periodically rigid by (\ref{eq:3-0}).
\end{proof}

The following is an analogue of \cite[Theorem 3.2]{jj}.
Using the lemmas collected so far, its proof is now a direct adaptation of that of \cite[Theorem 3.2]{jj}.
\begin{theorem} \label{lem:5}
Let $G$ be unbalanced and suppose that it has no $(0,2)$-block.
Then $G$ is 2-connected and redundantly periodically rigid if and only if $G$ is M-connected.
\end{theorem}
\begin{proof}
Suppose first that $G$ is M-connected.
By Lemma~\ref{lem:7} $G$ is 2-connected.
Also  $G$ is periodically rigid by Lemma \ref{lem:3}.
Since every edge is included in a circuit, we have $r_2(G-e)=r_2(G)$ for every edge $e\in E(G)$, meaning that $G-e$ is periodically rigid.
Thus $G$ is redundantly periodically rigid.

Conversely, suppose that $G$ is 2-connected and redundantly periodically rigid.
Suppose further for a contradiction that $G$ is not $M$-connected.
Then $E(G)$ can be decomposed into M-connected components $E_1,\ldots, E_t$ with $t\geq 2$.
By Lemmas~\ref{lem:3} and \ref{lem:4},
we have \(r_2(E_i)=2|V(E_i)|-3\) if \(E_i\) is balanced and
\(r_2(E_i)=2|V(E_i)|-2\) otherwise.
Let $a_i$ be the number of vertices in $G[E_i]$ which are shared by other $G[E_j]$.
With this notation, we have
\begin{align}
\nonumber
&2|V(G)|-2\geq r_2(E(G))=\sum_{i=1}^t r_2(E_i) \\
&= \sum_{E_i: \textrm{balanced}} (2|V(E_i)|-3) +\sum_{E_i:  \textrm{unbalanced}} (2|V(E_i)|-2) \nonumber \\
&=\sum_{i=1}^t (2|V(E_i)|-a_i)+\sum_{E_i: \textrm{balanced}} (a_i-3) +\sum_{E_i: \textrm{unbalanced}} (a_i-2). \label{eq:lem5}
\end{align}

Since $G$ is redundantly periodically rigid, every $M$-connected component contains a circuit.
Hence, for each balanced $E_i$, we have $|V(E_i)|\geq 4$ (since the smallest balanced circuit is $K_4$). This means that, if $a_i\leq 2$ for a balanced $E_i$, then $G[E_i]$ is a $(0,2)$-block in $G$. Hence we have $a_i\geq 3$ for each balanced $E_i$.
On the other hand, for each unbalanced $E_i$, we have $a_i\geq 2$ since $G$ is 2-connected.
Therefore the last two terms in (\ref{eq:lem5}) are nonnegative.
Observe finally that \(\sum_{i=1}^k (2|V(E_i)|-a_i)\geq 2|V(G)|\),
since each vertex of \(G\) contained in a unique $E_i$ is counted twice while  the rest of the vertices are counted at least twice.
Therefore (\ref{eq:lem5}) is at least $2|V(G)|$, which is a contradiction.
\end{proof}

\subsection{Proof of Lemma~\ref{lem:comb2}}
We first solve the following main case of the proof of Lemma~\ref{lem:comb2}.
\begin{lemma}
\label{lem:comb2pre}
Let $G$ be a 2-connected and redundantly periodically rigid graph  that has no $(0,2)$-block. Then for any vertex $v$ with degree $3$ in $G$,  $G_v$ is 2-connected and redundantly periodically rigid and has no $(0,2)$-block.
\end{lemma}
\begin{proof}
The 2-connectivity of $G_v$ follows from the 2-connectivity of $G$.
Let $E_v$ be the set of edges incident with $v$ in $G$, and assume that those edges are directed from $v$.
Let $F$ be the set of edges $e_1\cdot e_2$ over all pairs of nonparallel edges $e_1, e_2$ in $E_v$. Recall that $G_v$ is obtained from $G$ by removing $v$ and adding $F$, 
where the label of $e_1\cdot e_2$ is defined to be $\psi(e_1)^{-1}\psi(e_2)$.
Thus, we have that
\begin{equation}
\label{eq:comb2}
\text{$E_v\cup F$ is a circuit.}
\end{equation}

We first show that $G_v$ has no $(0,2)$-block.
Suppose to the contrary that $G_v$ contains a $(0,2)$-block $H$.
Take a maximal such $H$. 
We split the proof into four cases depending on how $H$ intersects with $F$ as follows.
(Case 1) If $|E(H)\cap F|=0$, then $H$ is also a $(0,2)$-block of $G$, a contradiction.
(Case 2)  If $|E(H)\cap F|=1$, then $B(H)$ is exactly the end vertices of the edge $e$ in $E(H)\cap F$.
 Hence $H\cap G$ would be a $(0,2)$-block in $G$ with
 $V(H\cap G)=V(H)$ and $B(H\cap G)=B(H)$. This is a contradiction.
(Case 3) If $F\subseteq E(H)$, then $|F|=3$ and it is easy to see that $(H+E_v)\cap G$ is a $(0,2)$-block in $G$,  again a contradiction.   
(Case 4) The remaining case is when $|N_G(v)|=3$ and $|E(H)\cap F|=2$.
Recall that $F$ forms a balanced triangle if $|N_G(v)|=3$. Hence $H+F$ would be  balanced and hence be a $(0,2)$-block in $G_v$, which contradicts the maximality of $H$.
This completes the proof that $G_v$ has no $(0,2)$-block.

To see that $G_v$ is redundantly periodically rigid, we first check that $G_v$ is periodically rigid.
Indeed, since $G$ is redundantly periodically rigid, $G-e'$ is periodically rigid for an edge $e'$ incident to $v$.
Then $v$ has degree two in $G-e'$, and by Lemma~\ref{lem:0extension} $G-v$ is periodically rigid. 
As $G_v$ is obtained from $G-v$ by adding $F$, $G_v$ is periodically rigid. 

Now take any edge $e\in E(G_v)$. Our goal is to show that $G_v-e$ is periodically rigid.
We take an edge $f$ from $E_v$.
Suppose first that there is a circuit \(C\) in \(G+F\) with \(f\in C\) and \(e\not\in C\).
Since $C\subseteq E(G)+F-e$, we have ${\rm cl}_2(E(G)+F-e)={\rm cl}_2(E(G)+F-e-f)$.
On the other hand, since $G$ is redundantly periodically rigid, ${\rm cl}_2(E(G)+F)={\rm cl}_2(E(G)+F-e)$.
Thus, ${\rm cl}_2(E(G)+F)={\rm cl}_2(E(G)+F-e-f)$,
and we get $r_2(G+F-e-f)=r_2(G+F)=2|V(G)|-2$.
Since $v$ has degree two in $G+F-e-f$, we get $r_2(G_v-e)=2|V(G)|-2-2=2|V(G_v)|-2$. Thus,  $G_v-e$ is periodically rigid.

So we may suppose that every circuit in \(G+F\) that contains \(f\) also contains \(e\).
Then by (\ref{eq:comb2}), $e\in F$.
Suppose that $G_v-e$ is not periodically rigid.
\begin{claim}\label{claim:mc}
$G_v-e$ contains an $M$-connected component $D'$ with $D'\cap (F-e)=\emptyset$.
\end{claim}
\begin{proof}
We split the proof into two cases depending on the size of $N_G(v)$.

\medskip
\noindent
Case 1: $|N_G(v)|=3$.
We denote $N_G(v)=\{v_1, v_2, v_3\}$, and assume that $e$ connects  $v_1$ and $v_2$.
For $i=1,2$, let $e_i$ be the edge in $F$ that connects  $v_i$ and $v_3$,
and let $D_i$ be the $M$-connected component in $G_v-e$ that contains $e_i$.

We first show that $D_1\neq D_2$.
To see this, suppose to the contrary that $D_1=D_2$. 
Then we have $e\in {\rm cl}_2(D_1)$.
Indeed, if $D_1$ is unbalanced, then ${\rm cl}_2(D_1)$ contains every edge on $V(D_1)$ by  Lemma~\ref{lem:3},
implying $e\in {\rm cl}_2(D_1)$.
On the other hand, if $D_1$ is balanced, then  by Lemma~\ref{lem:4} ${\rm cl}_2(D_1)$ contains every edge $e'$ on $V(D_1)$ for which $D_1+e'$ is balanced,
and $D_1+e$ is indeed balanced by the balancedness of $\{e, e_1, e_2\}$ and Lemma~\ref{lem:disconnected}.
Hence $e\in {\rm cl}_2(D_1)$.
%
This implies that ${\rm cl}_2(E(G_v))={\rm cl}_2(E(G_v-e))$ and hence $G_v$ is not periodically rigid as $G_v-e$ is not. 
This contradiction implies $D_1\neq D_2$.


Suppose, contrary to the claim,  that $D_1$ and $D_2$ are all the $M$-connected components in $G_v-e$.
By Lemma~\ref{lem:4} and Lemma~\ref{lem:3}, $r_2(D_i)=2|V(D_i)|-3+\delta(D_i)$, where $\delta(D_i)=0$ if it is balanced and  $\delta(D_i)=1$ otherwise.
Since $G_v-e$ is not periodically rigid, $2|V(G_v)|-3\geq r_2(G_v-e)=r_2(D_1)+r_2(D_2)=
2(|V(D_1)|+|V(D_2)|)-6+\delta(D_1)+\delta(D_2)$,
implying
$2|V(D_1)\cap V(D_2)|+\delta(D_1)+\delta(D_2)\leq 3$.
Since $V(D_1)\cap V(D_2)\neq \emptyset$, we obtain
\[
|V(D_1)\cap V(D_2)|=1 \text{ and } \delta(D_1)+\delta(D_2)\leq 1.
\]
The first equation implies that  $v_3$ is a cut vertex in $G_v-e$. 
The second inequality implies that at least $D_1$ or $D_2$ is balanced.
Without loss of generality we assume that $D_2$ is balanced.
If $|V(D_2)|\geq 3$, then $D_2$ is a $(0,2)$-block in $G_v$ since $\{v_2, v_3\}$ forms a cut, which contradicts that $G_v$ has no $(0,2)$-block. On the other hand, if $|V(D_2)|=2$, then  $D_2=\{e_2\}$, and the edges $e$ and $e_2$ induce a $(0,2)$-block in $G_v$, which is again a contradiction.
This contradiction completes the proof for Case 1.

\medskip
\noindent
Case 2: $|N_G(v)|=2$.
 Let $e_1$ be the edge in $F$ different from $e$.
Let $D_1$ be the $M$-connected component in $G_v-e$ that contains $e_1$.
If $D_1$ is unbalanced, then $D_1$ is periodically rigid by Lemma~\ref{lem:3}, and hence $e\in {\rm cl}_2(D_1)$.
This however implies that $G_v$ is not periodically rigid, which is a contradiction.
Thus $D_1$ is balanced.

If  $D_1$ is the only $M$-component in $G_v-e$, then $G[D_1]$  or $G[D_1-e_1]$ would be a $(0,2)$-block in $G$ by $|V(D_1)|\geq 3$. Thus $G_v-e$ has an $M$-component  different from $D_1$.
\end{proof}

Let $D'$ be an $M$-component in $G_v-e$ satisfying the condition of Claim~\ref{claim:mc}.
Let \(g\in D'\) be arbitrary.
By Theorem~\ref{lem:5} $G$ is $M$-connected, and hence $G+F$ is $M$-connected.
Hence $G+F$ contains a circuit $C_1$ with \(e,g\in C_1\).
Recall that $E_v\cup F$ is a circuit by (\ref{eq:comb2}).
 Since $e\in F$ and $g\notin E_v+F$, by the circuit elimination,
 we have a circuit $C_2\subseteq (C_1\cup (E_v+F))-e$ with $g\in C_2$.
 We claim
 \begin{equation}
 \label{eq:comb2-3}
 C_2\cap (E_v+e)=\emptyset.
 \end{equation}
 To see this recall that every circuit in $G+F$ containing $f$ also contains $e$.
 As $e\notin C_2$, we get $f\notin C_2$.
 Moreover, since $v$ has degree three in $G+F$, $f\notin C_2$ implies $E_v\cap C_2=\emptyset$,
 implying (\ref{eq:comb2-3}).

 By (\ref{eq:comb2-3}),  $C_2$ is a circuit in $G_v-e$.
However,  according to the construction of $C_2$,  we have $C_2\cap (E_v+F)\neq \emptyset$,
which means that $C_2$ intersects $F-e$ by (\ref{eq:comb2-3}).
Since $D'\cap (F-e)=\emptyset$ and $g\in C_2\cap D'$,
$C_2$ intersects more than one $M$-connected component in $G_v-e$,  which is a contradiction.
\end{proof}

\begin{proof}[Proof of Lemma~\ref{lem:comb2}]
By Lemma~\ref{lem:0extension}, the minimum degree of $G$ is at least three.
Let $v$ be a vertex of degree three, and suppose without loss of generality that all the edges incident to $v$ are directed from $v$. First we show that $v$ is $2$-nondegenerate, i.e. that $\bigcup_{v_i\in N_G(v)}\{p(v_i)+L(\psi(e)):\, e=vv_i\in E(G)\}$ is affinely independent for any generic $p:V(G)\to \mathbb{R}^2$. If $v$ has three distinct neighbors, then this is clearly true. 
If $v$ has exactly two neighbors, any parallel edges are not identical, and hence the parallel edges form an unbalanced cycle.  Thus the statement is again true.
Hence by 2-connectivity of $G$ we conclude that \(v\) is $2$-nondegenerate.

As shown in the proof of Lemma~\ref{lem:comb2pre}, $G-v$ is periodically rigid. Moreover,
the statement for $G_v$ has also already been proved in Lemma~\ref{lem:comb2pre}.
\end{proof}

\subsection{Proof of Lemma~\ref{lem:comb}}
The proof of Lemma~\ref{lem:comb} is rather involved, and it consists of  three major parts.
In Lemma~\ref{lem:8} we first show the existence of a degree three vertex in a minimally $M$-connected graph.
By this result, we may focus on the case when $G$ is not minimally $M$-connected, i.e., $E'=\{e\in E(G): G-e \text{ is $M$-connected}\}$ is nonempty.
By Theorem~\ref{lem:5}, $G-e$ is 2-connected and redundantly periodically rigid for every $e\in E'$ for which $G-e$ has no $(0,2)$-block.
Thus, what remains to show is that there is an edge $e\in E'$ such that $G-e$ has no $(0,2)$-block. 
We prove this by contradiction. A key lemma for this will be Lemma~\ref{lem:key}, which shows that, for any $(0,2)$-block $H_e$ in $G-e$ for $e\in E'$, $H_e$ contains an edge $f$ in $E'$.
The proof is completed by first taking $e\in E'$ such that $|V(H_e)|$ is as small as possible and then showing that $G-f$ has no $(0,2)$-block for $f\in E(H_e)\cap E'$.

We start with the following lemma, which is an unbalanced version of  Lemma~\ref{lem:jordan}.
The proof strategy is similar, but its proof is technically more involved.
\begin{lemma} \label{lem:8}
If $G$ is unbalanced and minimally M-connected, then $G$ has a vertex of degree three.
\end{lemma}
\begin{proof}
Let $C_1,\ldots, C_t$ be an ear decomposition of  $E(G)$.
The statement trivially follows from the edge count of $G$ if $t=1$. Hence we assume $t\geq 2$.
Let $D_i=\bigcup_{j=1}^i C_j$. We will use the notation $V=V(G)$ and $V'=V(D_{t-1})$.
Our goal is to prove that the average degree of $V\setminus V'$ (denoted by \(d_{avg}(V\setminus V')\)) is less than 4, implying that $V\setminus V'$ must contain a vertex of degree three.
The proof is split into two cases depending on whether $D_{t-1}$ is balanced or not.

\noindent
\textbf{Case 1}:
Suppose that $D_{t-1}$ is unbalanced. Then it follows from Lemma~\ref{lem:3} that $D_{t-1}$ and $D_{t-1}\cup C_t$ are periodically rigid.
Using Lemma \ref{lem:EDfacts}(c) we have
\begin{equation}
\label{eq:lem8-1}
|C_t\setminus D_{t-1}|-1=r_2(D_{t-1}\cup C_t)-r_2(D_{t-1})=2|V|-2-(2|V'|-2)=2|V\setminus V'|.
\end{equation}

Since $D_{t-1}$ is periodically rigid, $K(V',\Gamma)\subseteq {\rm cl}_2(D_{t-1})$.
Therefore, since $G$ is minimally $M$-connected, no edge in $C_t\setminus D_{t-1}$ is induced by $V'$.
This implies $V\setminus V'\neq \emptyset$, and by Lemma~\ref{lem:7}  we further have
\begin{equation}
\label{eq:lem8-2}
d(V\setminus V',V')\geq 3.
\end{equation}

Combining (\ref{eq:lem8-1}) and (\ref{eq:lem8-2}), we get
$$d_{avg}(V\setminus V')=\frac{2|C_t\setminus D_{t-1}|-d(V\setminus V',V')}{|V\setminus V'|}=\frac{4|V\setminus V'|+2-d(V\setminus V',V')}{|V\setminus V'|}<4.$$

\noindent
\textbf{Case 2}:
Suppose that $D_{t-1}$ is balanced. We may assume that $C_t$ is balanced (otherwise, by Lemma~\ref{lem:EDfacts}, we can start the ear decomposition with the unbalanced circuit $C_{t}$, and hence  we are back in Case~1). We prove  that $V\setminus V'\neq \emptyset$.

Since $D_{t-1}$ is balanced, we may assume $\psi(e)=\textrm{id}$ for every $e\in D_{t-1}$.
Let $X=\{e\in D_t|\, \psi(e)\neq \textrm{id}\}$.
Since $D_t$ is unbalanced, we have $\emptyset\neq X\subseteq C_t\setminus D_{t-1}$.
However, since $C_t$ is balanced, $G[C_t]-X$ must be disconnected.
This in turn implies $|X|\geq 3$ by Lemma~\ref{lem:7}, and hence
\begin{equation}
\label{eq:lem8-4}
|C_t\setminus D_{t-1}|\geq 3.
\end{equation}
On the other hand, by Lemma~\ref{lem:3} $D_{t-1}\cup C_t$ is periodically rigid, 
and it follows from Lemma~\ref{lem:4} that $D_{t-1}$ is rigid since it is balanced and M-connected.
Thus, by Lemma \ref{lem:EDfacts}(c) we have
\begin{equation}
\label{eq:lem8-3}
|C_t\setminus D_{t-1}|-1=r(D_{t-1}\cup C_t)-r(D_{t-1})=2|V|-2-(2|V'|-3)=2|V\setminus V'|+1.
\end{equation}
Combining (\ref{eq:lem8-4}) and (\ref{eq:lem8-3}), we get $V\setminus V'\neq \emptyset$.
By Lemma~\ref{lem:7} we again have (\ref{eq:lem8-2}).

Let $F$ be the set of edges in $C_t\setminus D_{t-1}$ induced by $V'$.
If $F\neq\emptyset$, then
we have
\begin{align*}
d_{avg}(V\setminus V')
&=\frac{2|C_t\setminus (D_{t-1}\cup F)|-d(V\setminus V',V')}{|V\setminus V'|}
=\frac{2|C_t\setminus D_{t-1}|-2|F|-d(V\setminus V',V')}{|V\setminus V'|} \\
&=\frac{4|V\setminus V'|+4-2|F|-d(V\setminus V',V')}{|V\setminus V'|}<4,
\end{align*}
where the third equation follows from (\ref{eq:lem8-3}) and the last inequality follows from (\ref{eq:lem8-2}) and $F\neq 0$.

Hence we suppose $F=\emptyset$ and every edge in $C_t\setminus D_{t-1}$ is incident to $V\setminus V'$.
By Lemma~\ref{lem:disconnected}, the graph $(V(C_t)\cap V(D_{t-1}), C_t\cap D_{t-1})$ is disconnected.
Since every edge in $C_t\setminus D_{t-1}$ is incident to $V\setminus V'$,
the disconnectivity of $(V(C_t)\cap V(D_{t-1}), C_t\cap D_{t-1})$ implies that
$V\setminus V'$ is a cut set in $G[C_t]$.
Hence Lemma~\ref{lem:7} implies  $|V\setminus V'|\geq 2$.
Since $C_t$ is balanced, the number of edges of $C_t$ induced by $V\setminus V'$ is at most $2|V\setminus V'|-3$, and we obtain
\[
|C_t\setminus D_{t-1}|\leq 2|V\setminus V'|-3+d(V\setminus V',V').
\]
Combining this with (\ref{eq:lem8-3}), we get \(d(V\setminus V',V')\geq5\). Applying (\ref{eq:lem8-3}) again gives
\[
d_{avg}(V\setminus V')=\frac{2|C_t\setminus D_{t-1}|-d(V\setminus V',V')}{|V\setminus V'|}
\leq\frac{4|V\setminus V'|-1}{|V\setminus V'|}<4.
\]
This completes the proof.
\end{proof}

The next target is to prove Lemma~\ref{lem:9} and Lemma~\ref{lem:10}, which will be needed for the proof of Lemma~\ref{lem:key}.
Lemmas~\ref{lem:9} and~\ref{lem:10} may be considered as periodic versions of the 2-sum lemma and the cleaving lemma given in \cite{jj}.
For the proofs we first need three technical lemmas.

\begin{lemma} \label{lem:2sum}
Let $F_1$ and $F_2$ be edge sets in a $\Gamma$-labeled graph $G$. 
Suppose that $F_1$ is balanced, 
$|V(F_1)\cap V(F_2)|=2$, and $F_1\cap F_2$ has exactly one edge, denoted by $f$.
\begin{itemize}
\item[(i)] Suppose that  $F_1-f$ has a path between the end vertices of $f$.
Then $F_1\cup F_2 - f$ is balanced if and only if $F_2$ is balanced.
\item[(ii)] Suppose that $F_1$ and $F_2$ are circuits. Then $F_1\cup F_2-f$ is a circuit.
\end{itemize}
\end{lemma}
\begin{proof}
(i) Since $F_1$ is balanced, we may assume that the label of each edge in $F_1$ is identity.
If $F_2$ is balanced, then $F_1\cup F_2$ is balanced by Lemma~\ref{lem:disconnected}, and hence $F_1\cup F_2-f$ is balanced as well.
If $F_2$ is unbalanced, then it has an unbalanced cycle $C$. If $f\notin C$, it remains in $F_1\cup F_2 -f$, and $F_1\cup F_2 -f$ is unbalanced.
If $f\in C$, then the sum of the labels in $C-f$ is non-identity as $f$ has the identity label.
Concatenating $C-f$ with a path in $F_1-f$, we get an unbalanced cycle in $F_1\cup F_2-f$.

\medskip
%
(ii) Let $F=F_1\cup F_2-f$.
Let $\delta=0$ if $F_2$ is balanced, and otherwise $\delta=1$.
Since $F_1$ and $F_2$ are circuits,  we have $|F|=|F_1|+|F_2|-2=(2|V(F_1)|-2)+(2|V(F_2)|-2+\delta)-2=2|V(F_1\cup F_2)|-2+\delta$.
By (i), $\delta=0$ if $F$ is balanced, and otherwise $\delta=1$. 
Hence $F$ has the right number of edges to be a circuit.

Suppose that $F$ is not a circuit. Then there exists a proper edge subset $F'$ of $F$ 
with $|F'|\geq 2|V(F')|-2+\delta'$, where $\delta'=0$ if $F'$ is balanced, and otherwise $\delta'=1$.
Let $F_i' = F'\cap F_i$. Then $F'$ is the disjoint union of $F'_1$ and $F'_2$.
Also since each $F_i$ is a circuit, each $F'_i$ is nonempty.

Since $F_1'$ is balanced and independent (as $F_1'\subseteq F_1-f$), we have $|F_1'|\leq 2|V(F_1')|-3$.
Similarly, since  $F_2'$ is independent with $F_2'\subseteq F_2-f$, we have $|F_2'|\leq 2|V(F_2')|-3+\delta'$.
Moreover, since $F'$ is a proper edge subset of $F$, 
the equations do not hold simultaneously. 
(If  both equations hold, then we would have $F_i'=F_i-f$, as $F_i$ is a circuit, and hence $F'=F$ would follow.)
Thus $|F'|=|F_1'|+|F_2'|\leq 2|V(F'_1)|-3+2|V(F'_2)|-3+\delta'-1=2|V(F')|-3+\delta'$, which is a contradiction. 

Therefore, every proper edge subset $F'$ of $F$ satisfies $|F'|\leq 2|V(F')|-3+\delta'$, 
and $F$ is a circuit.
\end{proof}

Since our matroid becomes an ordinary generic rigidity matroid if the graph is balanced, 
by adapting notations, we have the following from Lemma~4.2 in \cite{bj} (see also Lemma 2.18 in \cite{jj}).
\begin{lemma}\label{lem:cleavage0}
Let $G$ be a balanced circuit, and let $H$ be a $(0,2)$-block in $G$ with $B(H)=\{a,b\}$ and $V(H)\neq V(G)$. 
Then $H$ has no edge on $\{a,b\}$. 
Also let $f_0$ be the edge in $K(\{a,b\}, \Gamma)$ such that $H+f_0$ is balanced.  
Then $H+f_0$ and $G-I(H)+f_0$ are balanced circuits.
\end{lemma}

We prove the counterpart for unbalanced circuits.
\begin{lemma}\label{lem:cleavage1}
Let $G$ be an unbalanced  circuit and let $H$ be a $(0,2)$-block in $G$ with $B(H)=\{a,b\}$.
Then $H$ has no edge on $\{a,b\}$.
Also let $f_0$ be the edge in $K(\{a,b\}, \Gamma)$ such that $H+f_0$ is balanced. 
Then $H+f_0$ is a balanced circuit, and $G-I(H)+f_0$ is an unbalanced circuit.
\end{lemma}
\begin{proof}
Let $G_1$ be the subgraph of $G$ induced by $V(G)-I(H)$.
Suppose that $H$ has an edge between $a$ and $b$.
Then that edge is also in $G_1$, and we have 
$|E(G)|=|E(H)|+|E(G_1)|-1\leq 2|V(H)|-3+2|V(G_1)|-2-1=2|V(G)|-2$, contradicting that $G$ is an unbalanced circuit.
Thus $H$ has no edge between $a$ and $b$.
The same counting argument also shows that 
$|E(H)|=2|V(H)|-3$ and $|E(G_1)|=2|V(G_1)|-2$. 

Since $H$ is balanced, there is exactly one edge $f_0$ in $K(\{a,b\}, \Gamma)$ whose addition to $H$ keeps the balancedness.
By Lemma~\ref{lem:2sum}(i) $G_1+f_0$ is unbalanced. 
Moreover, $f_0\notin E(G_1)$ since otherwise $H+f_0$ would be a circuit properly contained in $G$.
Hence $|E(H+f_0)|=2|V(H)|-2$ and $|E(G_1+f_0)|=2|V(G_1)|-1$.

If $H+f_0$ is not a circuit, then $H$ has a proper edge subset $F$ with $|F|=2|V(F)|-3$ and $\{a,b\}\subseteq V(F)$.
Then $|F\cup E(G_1)|=2|V(F)|-3+2|V(G_1)|-2=2|V(F)\cup V(G_1)|-1$, which contradicts that $G$ is a circuit.
Therefore $H+f_0$ is a circuit.

If $G_1+f_0$ is not a circuit, then $G_1$ has a proper edge subset $F$ such that $F+f_0$ is a  circuit.
Since $H+f_0$ is a circuit, by Lemma~\ref{lem:2sum}(ii), $E(H)\cup F$ would be a circuit properly contained in $G$, a contradiction.
Therefore $G_1+f_0(=G-I(H)+f_0)$ is a circuit.
\end{proof}

We now extend Lemma~\ref{lem:cleavage0} to general $M$-connected graphs.

\begin{lemma}
\label{lem:9}
Let $G$ be unbalanced and $M$-connected, and let $H$ be a $(0,2)$-block in $G$ with $B(H)=\{a,b\}$.
Also let $f_0$ be the edge in $K(\{a,b\}, \Gamma)$ (which may exist in $G$)
such that $H+f_0$ is balanced.
Then $H+f_0$ is balanced $M$-connected and $G-I(H)+f_0$ is unbalanced $M$-connected.
\end{lemma}
\begin{proof}
Since $H$ is balanced, we may suppose that the label of each edge in $H$ is identity. 
Then $f_0$ is the edge in $K(\{a,b\}, \Gamma)$ with the identity label.

Let $G_1=G-I(H)$. By Lemma~\ref{lem:2sum}(i) and the unbalancedness of $G$, $G_1+f_0$ is unbalanced.
It remains to show that $H+f_0$ and $G_1+f_0$ are $M$-connected.


Take any edge $e$ in $H-f_0$ and any edge $e'$ in $G_1-f_0$.
Since $G$ is $M$-connected, $G$ contains a circuit $C$ containing $e$ and $e'$.
By Lemma~\ref{lem:cleavage0} and Lemma~\ref{lem:cleavage1}, $C$ does not contain $f_0$, and 
$C\cap E(H)+f_0$ and $C\cap E(G_1)+f_0$ are both circuits. 
Thus $e\sim f_0$ in $H+f_0$, and $e'\sim f_0$ in $G_1+f_0$.
Since $\sim$ is an equivalence relation, $H+f_0$ and $G_1+f_0$ are $M$-connected.
\end{proof}
%
%
%

The edge $f_0$ in Lemma~\ref{lem:9} is called the {\em cleaving edge} for the $(0,2)$-block $H$, and $H+f_0$ is called the {\em cleavage graph} of $H$.

\begin{lemma}\label{lem:10}
Let $G$ be $M$-connected, $H$ be a $(0,2)$-block in $G$, and $f_0$ be the cleaving edge for $H$.
If $H+f_0-f$ is $M$-connected for some $f\in E(H-f_0)$, then $G-f$ is $M$-connected.
\end{lemma}
\begin{proof}
Let $G_1=G-I(H)$.
By Lemma~\ref{lem:9},  $H+f_0$ is balanced and $G_1+f_0$ is $M$-connected.
Note also that $f\notin E(G_1+f_0)$ since $H-f_0$ has no edge on $\{a,b\}$ by the balancedness of $H+f_0$.

Take any edge $e\in E(H-f)$ and any edge $e'\in E(G_1)$.
Since $H+f_0-f$ (resp., $G_1+f_0$) is  $M$-connected, 
it contains a circuit $C$ (resp., $C'$)  with $e, f_0\in C$ (resp., $e', f_0\in C'$).
By Lemma~\ref{lem:2sum}(ii), $C\cup C'$ is a circuit.
Since $C\cup C'$ is in $G-f$, $e\sim e'$ in $G-f$.
Since $\sim$ is an equivalence relation, this in turn implies (by $E(H-f)\neq \emptyset\neq E(G_1)$) that every edge is related to each other in $G-f$. 
In other words, $G-f$ is $M$-connected.
\end{proof}

We are now ready to prove the following key lemma.

\begin{lemma}
\label{lem:key}
Let $G$ be an unbalanced $M$-connected graph that has no $(0,2)$-block and its  minimum degree is at least four. Suppose that $G-e$ is $M$-connected but $G-e$ has a $(0,2)$-block $H_e$ for some $e\in E(G)$.
Then $H_e-f_0$ contains an edge $f$ such that $G-f$ is $M$-connected, where $f_0$ is the cleaving edge for $H_e$.

Moreover, if an end vertex of $e$ is not in $I(H_e)$, then $f$ can be taken such that $H_e-f-f_0$ is 2-connected.
\end{lemma}
\begin{proof}
Let $u$ and $v$ denote the end vertices of $e$.
Since $H_e$ is balanced and $G$ has no $(0,2)$-block, we may suppose that
\begin{equation}
\label{eq:key1}
\text{the label of each edge in $H_e$ is the identity, and that $e$ has a non-identity label if $u,v\in V(H_e)$.}
\end{equation}

Let $\{a,b\}=B(H_e)$ and let  $G'=H_e+f_0$ be the cleavage graph for $H_e$.
At least one end vertex of $e$ is contained in $I(H_e)$, since otherwise $H_e$ would be a  $(0,2)$-block in $G$. Hence we may assume $u\in I(H_e)$. 
As $G'$ is (balanced) $M$-connected by Lemma~\ref{lem:EDfacts} and Lemma~\ref{lem:9}, 
we can take an ear decomposition $C_1,\dots, C_t$ of $E(G')$ such that $f_0\in C_1$ and $u\in V(C_1)$.

We first solve the case when an end vertex of $e$ is not in $I(H_e)$.
We first remark that $t>1$. Otherwise  $C_1$ contains at least four vertices of degree three in $C_1$,
one of which is also a degree three vertex in $G$ since an end vertex of $e$ is not in $I(H_e)$.
This contradicts that the minimum degree of $G$ is at least four.
Thus $t>1$. By Lemma~\ref{lem:jordan} and the minimum degree condition for $G$, we have $|\tilde{C}_t|=1$.
Let $f$ be the edge in $\tilde{C}_t$. Then $H_e+f_0-f$ is $M$-connected.
By Lemma~\ref{lem:10}, $G-f$ is also $M$-connected.
It remains to show that $H_e-f-f_0$ is 2-connected.
Suppose that $H_e-f-f_0$ is not 2-connected.
Then it is not (balanced) rigid, since every rigid graph is 2-connected.
This in turn implies that $H_e-f+f_0$ is not $M$-connected by Lemma~\ref{lem:4}, a contradiction.
This completes the proof in the case where  an end vertex of $e$ is not in $I(H_e)$.

The difficult case is when  both end vertices of $e$ are contained in $I(H_e)$.
We may assume that $|\tilde{C}_t|>1$, for otherwise  the edge in $\tilde{C}_t$ has the desired property by Lemma~\ref{lem:10}.
\begin{claim}
\label{claim:key1}
The following holds:
\begin{description}
\item[(i)] If $t=1$, then $u, v, a, b$ are distinct,  and they are  exactly the vertices of degree three in $G'$;
\item[(ii)] If $t>1$, then $V(C_t)\setminus V(\bigcup_{i=1}^{t-1} C_i)=\{v\}$, and $v$ has degree three in $G'$.
\end{description}
\end{claim}
\begin{proof}
If $t=1$, then there are at least four vertices of degree three in $G'$ since $E(G')$ is a balanced circuit.
Since the minimum degree of $G$ is at least four, we  have (i).

Suppose $t>1$, and suppose also that (ii) does not hold. Then by Lemma~\ref{lem:jordan} there is a vertex $w$ in $V(C_t)-V(\bigcup_{i=1}^{t-1} C_i)$ other than $v$ that has degree three in $G'$.  Since $\{u,a, b\}\subseteq V(C_1)$, $w$ is  distinct from them. Hence $w$ has degree three even in $G$, contradicting the minimum degree condition of $G$.
\end{proof}
Claim~\ref{claim:key1} implies
\begin{equation}
\label{eq:key2-0}
d_{G}(v)=4.
\end{equation}
Since $v\in I(H_e)$,  we can take an edge $f$ in $G'$ incident to $v$.
We claim that $G-f$ has a circuit $C^*$ with $e\in C^*$ and $f\notin C^*$.
Indeed, since $G-e$ is $M$-connected, $G-e$ has a circuit $C$ with $f\in C$,
and $G$ has a circuit $C'$ with $e\in C'$ and $f\in C'$.
By the circuit elimination, we get $C^*\subseteq C\cup C'-f$ with $e\in C^*$.
By (\ref{eq:key2-0}) we have
\begin{equation}
\label{eq:key2}
d_{C^*}(v)=3 \text{ and } d_{G-f}(v)=3.
\end{equation}

\begin{claim}
\label{claim:key2}
$C^*$ is unbalanced with $\{a,b\}\subseteq V(C^*)$.
\end{claim}
\begin{proof}
If $C^*\subseteq E(H_e)+e$, then we would have $r(C^*)=r(C^*-e)+1$ by $e\in C^*$ and (\ref{eq:key1}), which contradicts that $C^*$ is a circuit.
Hence $C^*$ must contain at least one edge from $E(G)-E(H_e+e)$.
Thus the 2-connectivity of $C^*$ implies $\{a,b\}\subseteq V(C^*)$.

Suppose that $C^*$ is balanced.
Then every path between $u$ and $v$ in $C^*-e$ passes through $a$, since the concatenation of $e$ and a simple path between $u$ and $v$ avoiding $a$ is unbalanced by (\ref{eq:key1}).
Hence $a$ is a cut vertex in $C^*-e$, contradicting the rigidity of $G[C^*]$.
(Note that $u$ and $v$ are distinct from $a$ by $u, v\in I(H_e)$.)
\end{proof}

By Claim~\ref{claim:key2} and Lemma~\ref{lem:circuit_periodic_rigid}, $C^*$ is periodically rigid with $\{a,b\}\subseteq V(C^*)$, and $f_0\in {\rm cl}_2(C^*)$.
Hence $C^*+f_0$ contains a circuit $C^*_0$ with $f_0\in C^*_0$.
Note that $C_0^*-f_0\subset C^*$. 


For any $e_1, e_2\in E(G-f)$, we denote $e_1\sim e_2$ if $G-f$ has a circuit that contains $e_1$ and $e_2$.
In order to show the $M$-connectivity of $G-f$, we show that $e_1\sim e_2$ for any $e_1 ,e_2\in E(G-f)$.
Since $\sim$ is an equivalence relation, we just need to show $e\sim e'$ for any $e'\in E(G-f)$,
and this follows by showing that  $G-f$ has a circuit intersecting $e'$ and $C^*$ for each $e'\in E(G-f)$.
For $e'\in E(G-f)\setminus E(H_e-f_0)$ this can be rapidly shown as follows.
Since $G-I(H_e)+f_0$ is $M$-connected by Lemma~\ref{lem:9}, it contains a circuit $C_{e'}$ with $e', f_0\in C_{e'}$.
If $C_{e'}=C^*_0$, then we have $e'\in C^*$ by $C^*_0-f_0\subset C^*$.
If $C_{e'}\neq C^*_0$, then by the circuit elimination $C_{e'}\cup C^*_0$ contains a circuit intersecting $e'$ and avoiding $f_0$. 
This circuit has the desired property since $C^*_0-f_0\subset C^*$ and $f\notin C_{e'}\cup C^*_0$.

In the following discussion, we prove that for each $e'\in E(H_e-f)$ there is a circuit intersecting $e'$ and $C^*$.
The proof consists of two cases, depending on whether $t=1$ or $t>1$.
The case when $t=1$ is easily solved by the following claim.  

\begin{claim}
\label{claim:key3}
If $t=1$, then $V(H_e)\subseteq V(C^*)$.
\end{claim}
\begin{proof}
Let $X=V(C^*)\cap V(H_e)$ and $Y=V(H_e)\setminus X$.
Also let $k$ be the number of edges between $a$ and $b$ in $G$.
Note that $C_1=E(H_e)+f_0$ by $t=1$. (Recall that $C_1$ is the initial circuit in the ear decomposition $C_1,\dots, C_t$.)

By Claim~\ref{claim:key2} and $e\in C^*$,
$\{u,v,a,b\}\cap Y=\emptyset$.
Hence the edge set of $G$ induced by $Y$ is a proper subset of $C_1$. Hence
\begin{equation}
\label{eq:key4}
i_G(Y)\leq 2|Y|-3
\end{equation}
if $|Y|\geq 2$.
On the other hand, for $X$, we claim 
\begin{equation}
\label{eq:key5}
i_G(X)\geq 2|X|-3+k.
\end{equation}
To see this, recall first that $C^*$ is an unbalanced   circuit by Claim~\ref{claim:key2}. 
Also, since $V(C^*)\neq \{a,b\}$,  $C^*$ can contain at most two edges between $a$ and $b$.
Hence, if $V(C^*)\subseteq V(H_e)$, we have $i_G(X)\geq 2|X|-1+(k-2)=2|X|-3+k$, which is (\ref{eq:key5}).
Suppose $V(C^*)\setminus V(H_e)\neq \emptyset$.
Let $C_1^*$ be the set of edges in $C^*$ induced by $V(H_e)$, 
and $C_2^*$ be the set of edges in $C^*$ induced by $V(G)-I(H_e)$.
Then $C_1^*\cap C_2^*$ is the set of edges in $C^*$ on $\{a,b\}$.
Since $C_i^*\neq \emptyset$ and $C^*=C_1^*\cup C_2^*$, 
we have $|C_1^*|-|C_1^*\cap C_2^*|= |C^*|-|C_2^*|\geq 2|V(C^*)|-1-(2|V(C^*_2)|-2)=2|V(C^*)|-3$.
Since $i_G(X)\geq |C_1^*|+(k-|C_1^*\cap C_2^*|)$, we get (\ref{eq:key5}).

Finally we claim
\begin{equation}
\label{eq:key6}
i_G(X\cup Y)=2|X\cup Y|-2+k.
\end{equation}
To see this, recall that  $E(H_e)+f_0$ is a balanced circuit.
Also $X\cup Y$ induces $E(H_e)+e$ and the edges on $\{a,b\}$.
Thus, if $f_0\in E(G)$ then we have $i_G(X\cup Y)= (2|X\cup Y|-2)+1+(k-1)$,
and otherwise we have $i_G(X\cup Y)= (2|X\cup Y|-3)+1+k$.
Hence (\ref{eq:key6}) holds.

Combining (\ref{eq:key4})(\ref{eq:key5})(\ref{eq:key6}), we get
\[2i_G(Y)+d(X,Y)= i_G(Y)+i_G(X\cup Y)-i_G(X)\leq 4|Y|-2\quad \text{if $|Y|\geq 2$}, \]
and
\[d(X,Y)= i_G(X\cup Y)-i_G(X)\leq 2|Y|+1=3\quad \text{if  $|Y|=1$}.\]
Those imply that, if $Y\neq \emptyset$, then $G$ has a  vertex of degree three in $Y$.
By the minimum degree condition, we conclude that $Y=\emptyset$.
\end{proof}

Suppose that $t=1$.
Since $C^*$ is periodically rigid, $e'\in {\rm cl}_2(C^*)$ holds for any $e'\in E(H_e-f)$.
Hence $C^*+e'$ contains a circuit $C_{e'}$ intersecting $e'$ and $C^*$, which has the desired property.

Suppose that $t>1$.
By Claim~\ref{claim:key1}(ii) and (\ref{eq:key2}), every edge in $\tilde{C}_t-f$ is included in $C^*$.
Since $E(H_e)=\bigcup_{i=1}^t (C_i-f_0)$ and $f\in \tilde{C}_t$, it remains to solve the case when $e'\in C_i-f_0$ for some $i$ with $i<t$.
We solve it by induction on $i$.
Suppose $e'\in C_1-f_0$. 
If $C_1=C^*_0$, then $e'\in C^*$ by $C_0^*-f_0\subset C^*$.
Otherwise, since $f_0\in C_1$ from the definition of $C_1$, by the circuit elimination  there is a circuit $C_{e'}$ with $e'\in C_{e'}\subseteq C_1\cap C_0^*-f_0$.
This circuit is contained in $G-f$ and intersects $e'$ and $C^*$ (by $C^*_0-f_0\subset C^*$).
This solves the base case.
For the general case, let $e'\in C_i-f_0$. 
If $f_0\in C_i$, one can apply  exactly the same argument as in the base case.
Otherwise $C_i$ is contained in $G-f$ by $f\in \tilde{C}_t$.
Moreover by the definition of ear decomposition $C_i$ contains an edge $e''$ in $\bigcup_{j<i} C_j$.
By induction $e''\sim e$, and we get $e'\sim e''\sim e$. 
This completes the proof.
\end{proof}

The following lemma lists properties of $(0,2)$-blocks which we will use frequently in the following.
Most of them follow directly from the definition.
\begin{lemma}
\label{lem:cleavage}
Let $G$ be an unbalanced and $M$-connected graph having no $(0,2)$-block,  and let $e\in E(G)$.
Suppose that $G-e$ is unbalanced and $M$-connected, but $G-e$ has a $(0,2)$-block $H_e$.
Then the following hold.
 \begin{description}
\item[(a)] Any edge of $G-e$ induced by $V(H_e)$ is included in $H_e$ unless it is on the boundary $B(H_e)$, i.e., $E_{G-e}(V(H_e))-E_{G-e}(B(H_e))\subseteq E_{G-e}(H_e)$.
 \item[(b)] At least one end vertex of $e$ is included in $I(H_e)$.
 \item[(c)] If $H_e$ is an inclusionwise minimal $(0,2)$-block in $G-e$, then the cleavage graph of $H_e$ is 3-connected.
 \item[(d)] $|E(G-e)\setminus E(H_e)|\geq 2$.
 \item[(e)] $(G-e)-I(H_e)$ is connected.
 \end{description}
\end{lemma}
\begin{proof}
(a) This directly follows from the definition of $B(H_e)$.

(b) If both end vertices of $e$ are not in $I(H_e)$, then $H_e$ would be a $(0,2)$-block in $G$, a contradiction.

(c)  If the cleavage graph of $H_e$ is not 3-connected, then $G-e$ would have a proper subgraph of $H_e$ which is a $(0,2)$-block. This contradicts the minimality of $H_e$.

(d) Recall that $H_e$ is balanced.
Hence $G-e$ contains at least two edges which are not in $H_e$, since otherwise $G$ cannot be $M$-connected.

(e) Let $B(H_e)=\{a,b\}$. Suppose that $(G-e)-I(H_e)$ is disconnected.
Since $G-e$ is connected, $(G-e)-I(H_e)$ has two connected components $C_a$ and $C_b$ containing $a$ and $b$, respectively. In particular, there is no edge between $a$ and $b$ in $G-e$.
If any one of the two components is nontrivial, then $G-e$ has  a cut vertex.
However, since $G-e$ is $M$-connected, Lemma~\ref{lem:7} implies that $G-e$ is 2-connected, a contradiction.

Thus we may assume that both components are trivial. Then we have $V(H_e)=V(G)$.
By (a), every edge in $E(G-e)-E(H_e)$ lies on $\{a,b\}$. However, since $G-e$ has no edge between $a$ and $b$, we would have $E(G-e)=E(H_e)$, contradicting (d).
 \end{proof}

We are now ready to prove Lemma~\ref{lem:comb}.

\begin{proof}[Proof of Lemma \ref{lem:comb}]
We show that (i) holds if (ii) does not.
Hence suppose that the minimum degree of $G$ is at least four. 
For any edge $e\in E(G)$, $G-e$ is unbalanced since otherwise $r_2(G)=2|V(G)|-2 > 2|V(G)|-3 =r_2(G-e)$ would hold, contradicting the $M$-connectivity of $G$.
Hence, by Theorem~\ref{lem:5} it suffices to show that
$G$ has an edge $e$  such that 
\begin{itemize}
\item $G-e$ is $M$-connected, and 
\item $G-e$ has no $(0,2)$-block.
\end{itemize}
By Lemma~\ref{lem:8}, $G$ is not minimally $M$-connected.
In other words,  \(E'=\{e\in E: G-e \hbox{ is M-connected}\}\) is not empty.
We show that $G-e$ has no $(0,2)$-block for some  $e\in E'$.

Suppose that $G-{e'}$ has a $(0,2)$-block for every $e'\in E'$.
Let $H_{e'}$ be an inclusionwise minimal $(0,2)$-block in $G-e'$ for each $e'\in E'$,
and take $e\in E'$ such that $|V(H_{e})|$ is as small as possible.
Since the minimum degree of $G$ is at least four, Lemma~\ref{lem:key} implies that
$H_{e}$ contains an edge $f\in E'$, i.e.,  $G-f$ is $M$-connected.
Lemma~\ref{lem:key} further says that $f$ is not the cleaving edge of $H_e$, which implies
\begin{equation}
\label{eq:4-5-1-1}
\text{at least one end vertex of $f$ is in $I(H_e)$.}
\end{equation}

Let $\{a,b\}$ be the boundary $B(H_{e})$ of $H_e$ in $G-e$, and $\{x, y\}$ be the boundary $B(H_f)$ of $H_f$ in $G-f$.
The cleaving edges for $H_{e}$ and $H_f$ are denoted by $f_{ab}$ and $f_{xy}$, respectively.
By Lemma~\ref{lem:9} and Lemma~\ref{lem:cleavage}(c),
$H_{e}+f_{ab}$ and $H_f+f_{xy}$ are balanced $M$-connected and 3-connected.
Also, by Lemma~\ref{lem:key}, we may suppose that
\begin{equation}
\label{eq:4-5-1}
\text{$H_{e}-f-f_{ab}$ is 2-connected if an end vertex of $e$ is not in $I(H_{e})$.}
\end{equation}
In the subsequent discussion, we will frequently use the fact that 
$f\in E(H_e), f\notin E(H_f)$, and $e\notin E(H_e)$.
 
We first claim  the following two technical facts.
\begin{claim}
\label{claim:last1}
If $V(G)=V(H_f)$, then  $x$ and $y$ are contained in $I(H_e)$.
\end{claim}
\begin{proof}
Suppose $V(G)=V(H_f)$ but $x$ is not contained in $I(H_e)$.
Let $F=E(G-f)\setminus E(H_f)$.
By $V(G)=V(H_f)$ and $|B(H_f)|=2$, $F$ is the set of parallel edges on $\{x,y\}(=B(H_f))$.
Note also that $f$ is not induced by $\{x,y\}$ by Lemma~\ref{lem:cleavage}(b).

Since $f\in E(H_e)$ and $H_e+f_{ab}$ is 3-connected, $H_e$ contains a cycle $C_a$ (resp., $C_b$) that passes through $f$ and avoids $a$ (resp., $b$).
Since $x$ is not contained in $I(H_e)$, $C_a$ or $C_b$ avoids $x$,
which means that $C_a\cap F=\emptyset$ or $C_b\cap F=\emptyset$.
Without loss of generality assume $C_a\cap F=\emptyset$.
By $F=E(G-F)\setminus E(G_f)$, $C_a\subset H_f+f$.
Since $f\in C_a\subset H_e\cap (H_f+f)$ and $H_e$ and $H_f$ are  balanced, we conclude that
$H_f+f$ is balanced.
This however implies that $H_f+f$ is a $(0,2)$-block in $G$, which is a contradiction.
\end{proof}

\begin{claim}
\label{claim:last2}
If $H_e\cap H_f$ has a path between $a$ and $b$, then $H_e\cup H_f$ is balanced. 
\end{claim}
\begin{proof}
Since $H_e-f$ is balanced, we may assume that the label of each edge in $H_e-f$ is identity.

Suppose that $(H_e-f)\cup (H_f-e)$ is unbalanced. 
Take any unbalanced cycle $C$ in $(H_e-f)\cup (H_f-e)$.
As $H_e-f$ and $H_f-e$ are balanced, $C$ must intersect both 
$E(H_e-f)\setminus E(H_f-e)$ and $E(H_f-e)\setminus E(H_e-f)$,
and hence $C$ passes through $a$ and $b$ (by the definition of the boundary).
As $C$ is unbalanced, $C$ has a path between $a$ and $b$ that consists of edges in $H_f$ and the sum of the labels is non-identity. 
Concatenating this path with a path in $H_e\cap H_f$, we get an unbalanced cycle in $H_f$, contradicting the balancedness of $H_f$.
Thus $(H_e-f)\cup (H_f-e)$ is balanced. Hence we may further assume that the label of each edge in $(H_e-f)\cup (H_f-e)$ is identity.

As $H_e+f_{ab}$ is 3-connected, $H_e-f$ is connected.
Hence $H_e-f$ contains a path between the end vertices of $f$.
Since the label of each edge in the path is identity,  for $H_e$ to be balanced, the label of $f$ should be identity.
By the same argument, if $e\in E(H_f)$, then the label of $e$ should be identity.
Thus the label of each edge in $H_e\cup H_f$ is identity, 
and $H_e\cup H_f$ is balanced.
\end{proof}

We split the proof into four cases depending on the relative positions among $\{a, b, x, y\}$.

\medskip
\noindent
{\bf Case 1:} $x\in I(H_{e})$ and $y\in V(G)\setminus V(H_{e})$.
In this case $\{a,b\}$ is a cut of $G-e-f$ since $x$ and $y$ belong to different components in
$(G-e)-a-b$.
By Claim~\ref{claim:last1}, $\{x,y\}$ is also a cut of $G-e-f$.
(If $\{x,y\}$ is not a cut in $G-e-f$, then it is also not a cut in $G-f$ and hence $V(G)=V(H_f)$ follows.
Hence Claim~\ref{claim:last1} implies $x,y\in I(H_e)$, contradicting $y\in V(G)\setminus V(H_e)$.)
Now, from the fact that $\{a,b\}$ and $\{x,y\}$ are cuts of $G-e-f$,
 $G-e-f$ can be decomposed into four subgraphs $G_1, \dots, G_4$ by taking
$G_1=(H_{e}-f)\cap (H_f-e)$,
$G_2=(H_f-e)\cap (G-I(H_{e}))$,
$G_3=(G-I(H_f))\cap (H_{e}-f)$ and
$G_4=(G-I(H_f))\cap (G-I(H_{e}))$.
Then $G_1, G_2$ and $G_3$ are balanced, and each $G_i$ has at least two vertices.
Thus, $r_2(G-e-f)\leq \sum_{i=1}^4 r_2(G_i)\leq  \sum_{i=1}^4 2|V(G_i)|-11
\leq 2|V(G)|-3$. 
On the other hand, since $G-e$ is unbalanced $M$-connected, 
$r_2(G-e)=2|V(G)|-2$ by Lemma~\ref{lem:3}, and by the $M$-connectivity $r_2(G-e-f)=r_2(G-e)=2|V(G)|-2$.
This is a contradiction.

\medskip
\noindent
{\bf Case 2:} $x\in I(H_{e})$ and $y\in I(H_{e})$.
We first show the following:
\begin{equation}
\label{eq:4-5-3}
(H_e-f)\cup (H_f-e)=G-e-f.
\end{equation}
By the minimum choice of $V(H_e)$, $(G-e)-I(H_e)$ must contain a vertex $w\in I(H_f)$.
By Lemma~\ref{lem:cleavage}(e), $(G-e)-I(H_e)$ is connected,
and hence there is a path from $w$ to any vertex in $(G-e)-I(H_e)$.
Such a path avoids $\{x,y\}$ by $\{x,y\}\subseteq I(H_e)$ and avoids $f$ by (\ref{eq:4-5-1-1}), and hence we get $V(G)\setminus I(H_e)\subseteq I(H_f)$.
(Note that such a path starts from the interior of $H_f$ and remains in the interior, since it avoids the boundary $\{x,y\}$.)
Hence every edge induced by $V(G)\setminus I(H_e)$ is included in $H_f$ by Lemma~\ref{lem:cleavage}(a) and (\ref{eq:4-5-1-1}), implying (\ref{eq:4-5-3}).

We next prove
\begin{equation}
\label{eq:4-5-4}
\text{$(H_e-f)\cap (H_f-e)$ contains a path between $a$ and $b$.}
\end{equation}
Suppose not. By (\ref{eq:4-5-3}) one can see $a, b\in V(H_f)$. 
Let $K_x$ be the component of $(H_e-f)\cap (H_f-e)$ containing $x\in I(H_e)$. 
Since (\ref{eq:4-5-4}) does not hold, $K_x$ cannot contain both $a$ and $b$.
If $K_x$ contains neither $a$ nor $b$, then $K_x$ would be a connected component of $H_f-e$, 
and $H_f+f_{xy}$ is not 3-connected, which is a contradiction.
So we may suppose that $K_x$ contains $a$ but not $b$.
Then by Lemma~\ref{lem:cleavage}(a)(e) and (\ref{eq:4-5-1-1}), $a$ is a cut vertex in $H_f-e$.
   
Similarly, let $K_y$ be the component of $(H_e-f)\cap (H_f-e)$ containing $y\in I(H_e)$.
Then $K_y$ contains either $a$ or $b$.
If $K_y$ contains $a$ (i.e., $K_x=K_y$), then $a$ is a cut vertex in $H_f-e$ and it is still a cut vertex in $H_f-e+f_{xy}$. 
Hence $H_f+f_{xy}$ cannot be 3-connected, a contradiction.
If $K_y$ contains $b$, then by the 3-connectivity of $H_f+f_{xy}$, $H_f-e$ consists of three subgraphs $K_x, K_y$, and $K:=(H_f-e)-I(H_e)$ 
with $V(K_x)\cap V(K_y)=\emptyset, V(K_x)\cap V(K)=\{a\}, V(K_x)\cap V(K)=\{b\}$.
Since each subgraph is balanced and has at least two vertices, 
we have $r_2(H_f-e)\leq r_2(K_x)+r_2(K_y)+r_2(K)\leq 
2(|V(K_x)|+|V(K_y)|+|V(K)|-2)-5
=2|V(H_f-e)|-5$. 
On the other hand, since $H_f+f_{xy}$ is 3-connected and $M$-connected, it is redundantly rigid by Theorem 3.3 in \cite{jj}. 
This is a contradiction as $H_f$ is not rigid by $r_2(H_f-e)\leq 2|V(H_f-e)|-5$.

%

Thus we have (\ref{eq:4-5-4}).
By (\ref{eq:4-5-4}) and Claim~\ref{claim:last2}, $H_e\cup H_f$ is balanced.
By (\ref{eq:4-5-3}) and $f\in E(H_e)$, $H_e\cup H_f$ is either $G-e$ or $G$.
Hence $G-e$ or $G$ is balanced, which contradicts the unbalancedness of $G-e$ (which follows from the $M$-connectivity of $G$ as explained at the beginning of the proof.)

%
%

\medskip
\noindent
{\bf Case 3:} $x\in I(H_{e})$ and $y\in B(H_{e})$. Without loss of generality we assume $b=y$. 
We first remark that $a\in V(H_f)$. Indeed, due to the minimality of $V(H_e)$, $(G-e)-I(H_e)$ must contains at least one vertex $w$ from $I(H_f)$,
and hence if $a$ is not in $V(H_f)$ then $b$ would be a cut vertex in $(H_f-e)+f_{xy}$, which contradicts the 3-connectivity of $H_f+f_{xy}$.
 
We consider four  subgraphs $G_i\ (1\leq i\leq 4)$ of $G-e-f$ by taking
$G_1=(H_e-f)\cap (H_f-e)$,
$G_2=(H_e-f)\cap (G-I(H_f))$, 
$G_3=(G-I(H_e))\cap (H_f-e)$, and 
$G_4=(G-I(H_e))\cap (G-I(H_f))$.
Note that $G-e-f = \bigcup_{i=1}^4 G_i$. 
Also, by $a\in V(H_f)$, $a\in G_1$ and $a\in G_3$. 
As $\{a,b\}$ (resp., $\{x,y\}$) is the boundary of $H_e-f$ (resp., $H_f-e$) in $G-e-f$, we have
\[
\begin{array}{cc}
V(G_1\cap \bigcup_{i\neq 1} G_i)= \{x,b,a\}, &
V(G_2\cap \bigcup_{i\neq 2} G_i)= \{x,b\} \\
V(G_3\cap \bigcup_{i\neq 3} G_i)= \{a, b\}, &
V(G_4\cap \bigcup_{i\neq 4} G_i)= \{b\}.
\end{array}
\]
We have two subcases.

Case 3-1: Suppose that $G_1$ does not contain a path between $a$ and $b$.
If $G_1$ has no path between $a$ and $x$, then $b$ is a cut vertex in  $G_1\cup G_3(=H_f-e)$.
Then $b(=y)$ is still a cut vertex in  $H_f-e+f_{xy}$, and $H_f+f_{xy}$ cannot be 3-connected, a contradiction.
Therefore $G_1$ has a path between $a$ and $x$ (and has no path between $b$ and $x$).
This in turn implies that $x$ is a cut vertex in $G_1\cup G_2$.
Also $a$ is a cut vertex in $G_1\cup G_3(=H_f-e)$.
Hence, for the 3-connectivity of $H_f+f_{xy}$ at least one end vertex of $e$ is not in $I(H_e)$ (as $e$ should be incident to a vertex in $V(G_3)$).
Thus, by (\ref{eq:4-5-1}), $H_e-f-f_{ab}$ is 2-connected.
However, as $G_1\cup G_2=H_e-f$, this contradicts  the fact that  $x$ is a cut vertex in $G_1\cup G_2$.

Case 3-2: Suppose that $G_1$ contains a path between $a$ and $b$.
Then by Claim~\ref{claim:last2} $H_e\cup H_f$ is balanced.
Note also that $G-e$ is the union of $H_e\cup (H_f-e)$ and $G_4$ with intersection only at $b$. 
If $G_4$ consists of just one vertex $b$, then by $f\in E(H_e)$ $H_e\cup H_f$ would be either $G$ or $G-e$,
and the balancedness of $H_e\cup H_f$ contradicts the unbalancedness of $G-e$.
If $G_4$ contains more than one vertex, then $b$ is a cut vertex in $G-e$, and an end vertex of $e$ is in $V(G_4)\setminus \{b\}$.
This in turn implies that $b$ and another end vertex of $e$ is the boundary of $H_e\cup H_f$ in $G$.
 In other words, $H_e\cup H_f$ is a $(0,2)$-block in $G$, a contradiction.

\medskip
\noindent
{\bf Case 4:} $x\notin I(H_{e})$ and $y\notin I(H_{e})$.
We first prove $I(H_e)\cap I(H_f)\neq \emptyset$. 
To see this, recall that an end vertex of $f$ is in $I(H_f)$ by Lemma~\ref{lem:cleavage}(b), 
and this vertex is in $V(H_e)$ by $f\in E(H_e)$.
If this vertex belongs to $I(H_e)$, we are done.
If this vertex is in $B(H_e)$, then the other end vertex of $f$ is in $I(H_e)$ by (\ref{eq:4-5-1-1}). 
This vertex also belongs to $I(H_f)$, since $H_e$ has a path between the end vertices of $f$ internally avoiding $\{x, y\}$ by the 3-connectivity of $H_e+f_{ab}$ and $x, y\notin I(H_{e})$.

Thus we can take $v^*\in I(H_e)\cap I(H_f)$. We next prove 
that $H_e\cap H_f$ contains a path between $a$ and $b$.
Since $H_e+f_{ab}$ is 3-connected, $H_e$ contains a path from $v^*$ to $a$ (resp., $b$) avoiding $f$ and $b$ (resp., $a$).
By $x, y \notin I(H_e)$, such a path goes through the interior of $H_f$ (and the last vertex $a$ (resp., $b$) may be the boundary).
Hence such a path is in $H_e\cap H_f$, and the concatenation of those paths would be a desired path between $a$ and $b$ in $H_e\cap H_f$.

Therefore, by Claim~\ref{claim:last2},  $H_e\cup H_f$ is balanced,
and we may suppose that the label of each edge in $H_e\cup H_f$ is identity.

We next remark that 
\begin{equation}\label{eq:case4-2}
\text{ an end vertex of $f$ does not belong to $V(H_f)$.}
\end{equation}
Indeed, if both end vertices of $f$ belong to $V(H_f)$,
then $H_f+f$ would be  balanced since the label of each edge in $H_f$ and $f\in E(H_e)$ is identity. 
Moreover, the boundary of $H_f+f$ in $G$ is $\{x,y\}$, which means that $H_f+f$ would be a $(0,2)$-block of $G$, 
a contradiction if both end vertices of $f$ belong to $V(H_f)$
Thus we get (\ref{eq:case4-2}).

Let $u^*\in V(H_e)\setminus V(H_f)$ be an end vertex of $f$ shown by (\ref{eq:case4-2}).
We prove  
\begin{equation}\label{eq:case4-3}
\{x,y\}= \{a,b\}.
\end{equation}
To see this, suppose $\{x, y\} \neq \{a, b\}$.
Since $H_e+f_{ab}$ is 3-connected, $H_e$ contains a path from $v^*$ to $u^*$ avoiding $f$ and $a$ (resp., $b$).
Since $\{x,y\}\neq \{a,b\}$ and $u^*\notin \{x, y\}\subseteq V(H_f)$, one such path avoids $\{x,y\}$.
This however implies $u^*\in I(H_f)$ by $v^*\in I(H_f)$, contradicting  $u^*\in V(H_f)$.

Since $u^*\in V(H_e)\setminus V(H_f)$, by the minimality of $V(H_e)$, 
$V(H_f)\setminus V(H_e)$ is nonempty.
Hence $\{x,y\}$ is a cut in $H_f-e$ that separates $v^*$ and $V(H_f)\setminus V(H_e)$.
As $H_f+f_{xy}$ is 3-connected, we should have $e\in E(H_f)$.
This however implies that $H_e\cup H_f$ is a balanced graph whose boundary in $G$ is $\{x,y\}(=\{a,b\})$, i.e., a $(0,2)$-block in $G$.
This contradicts that $G$ has no $(0,2)$-block, and the proof is complete.
\end{proof}

\section{Concluding Remarks}
\subsection{Global rigidity of cylindrical/toroidal frameworks}
Given \(\ell\in \mathbb{R}^2\),  we consider the following equivalence relation of \(\mathbb{R}^2\):
\[\hbox{for }a,b\in\mathbb{R}^2,a\sim_1 b\hbox{ if and only if }a=b+z\ell\hbox{ for some }z\in\mathbb{Z}.\]
A flat cylinder \({\cal C}_{\ell}={\cal C}\) is obtained by factoring out \(\mathbb{R}^2\) with the relation \(\sim_1\). Similarly, given a pair of vectors \(\ell_1\ell_2\) of \(\mathbb{R}^2\), consider  an equivalence relation of \(\mathbb{R}^2\) by:
\[\hbox{for }a,b\in\mathbb{R}^2,a\sim_2 b\hbox{ if and only if }a=b+z_1\ell_1+z_2\ell_2\hbox{ for some pair }z_1,z_2\in\mathbb{Z}.\]
A flat torus \({\cal T}_{\ell_1,\ell_2}={\cal T}\) is obtained by factoring out \(\mathbb{R}^2\) with the relation \(\sim_2\). 

Given a straight-line drawing of an undirected graph  on ${\cal C}$ (resp.~on ${\cal T}$), we regard it as a bar-joint framework. Note that such frameworks on ${\cal C}$ (resp.~on ${\cal T}$) has a one-to-one correspondence with $1$-periodic (resp.~$2$-periodic) frameworks in $\mathbb{R}^2$ through $\sim_1$ (resp.~$\sim_2$).
Hence our combinatorial characterization of the global rigidity of periodic frameworks immediately implies a characterization of the global rigidity of cylindrical/toroidal frameworks.

The underlying combinations of a drawing on ${\cal C}$ (resp.~on ${\cal T}$) is captured by using a
$\mathbb{Z}$-labeled graph $(G,\psi)$ (resp.~$\mathbb{Z}^2$-labeled graph), 
where each label determines the geodesic between two end vertices.
A cylindrical framework (resp.~a toroidal framework) is defined as a pair $(G,\psi, p)$ of 
a $\mathbb{Z}$-labeled graph $(G,\psi)$ (resp.~$\mathbb{Z}^2$-labeled graph) and 
$p:V(G)\rightarrow {\cal C}$ (resp.~$p:V(G)\rightarrow {\cal T}$).
Note that a subgraph $H$ is balanced if and only if it is contractible on the surface. 
Hence Theorem~\ref{thm:main0} can be translated to Theorem~\ref{thm:cylinder} for cylindrical frameworks. For toroidal frameworks the statement becomes as follows.

\begin{theorem}\label{thm:torus}
A generic  framework  $(G,p)$ with $|V(G)|\geq 3$ on ${\cal T}$ is globally rigid if and only if $G$ is connected, each two-connected component of $G$ is redundantly rigid on ${\cal T}$, has no contractible subgraph $H$ with $|V(H)|\geq 3$ and $|B(H)|=2$, and has rank two.
\end{theorem}

\subsection{Open problems}
There are quite a few remaining questions. As mentioned in the introduction, an important challenging problem is to extend our results to more general settings of global rigidity of periodic frameworks, as was done for local rigidity in \cite{mt13,mt14}.
 
As mentioned above, Theorems~\ref{thm:main0} and \ref{thm:main1} may be viewed as a characterization of generic globally rigid bar-joint frameworks on a flat cylinder  or a flat torus, respectively. A natural open problem is to establish counterparts of Theorem~\ref{thm:jj} in other flat Riemannian manifolds. (Note that for frameworks on the cylinder and other surfaces with edge lengths measured by the standard Euclidean distance in $\mathbb{R}^3$, generic global rigidity has been investigated in \cite{jmn,jn,jn1}.)

A similar question would be about the global rigidity of frameworks on a flat cone.
 Since a flat cone with cone angle $2\pi / n$ is the quotient of $\mathbb{R}^2$ by an $n$-fold rotation, 
the global rigidity of such frameworks can be understood by the global rigidity of frameworks in $\mathbb{R}^2$ with $n$-fold rotational symmetry (under the given symmetry constraints). 
 The corresponding local rigidity question has been studied in \cite{jkt,mt15} and an extension of Laman's theorem is known. 
 Since the space of trivial motions is only of dimension one in this case (only rotations are trivial), an unbalanced rigidity circuit $G$ satisfies the count $|E(G)|=2|V(G)|$. Therefore, we cannot guarantee the existence of a vertex of degree three in $G$. This constitutes the key obstacle in applying our current proof method to this problem.

We may also ask about the global rigidity of finite frameworks with other point group symmetries. 
However, for the same reason, it also remains open to characterize the symmetry-forced global rigidity of finite bar-joint frameworks in $\mathbb{R}^2$ that are generic modulo reflection or dihedral symmetry. In fact, it is currently not even known whether symmetry-forced global rigidity is a generic property for any point group in dimension 2.  Necessary redundant rigidity and connectivity conditions, however, may be obtained in a similar fashion as described in Section~\ref{sec:nec}.

\section*{Acknowledgement}

We would like to thank the two anonymous referees for their careful and thoughtful review and for providing a number of valuable suggestions for improvement.


\end{document}